\documentclass[12pt,leqno]{amsart}
\usepackage{amssymb}
\usepackage[all]{xy}

\def\E{{\mathcal{E}}}
\def\F{{\mathcal{F}}}

\def\M{{\mathcal{M}}}

\def\cO{{\mathcal{O}}}

\def\PP{{\mathbb{P}}}
\def\Q{{\mathbb{Q}}}

\def\Z{{\mathbb{Z}}}

\def\FB{{\mathrm{FB}}}
\def\Hilb{{\mathrm{Hilb}}}

\def\Im{{\mathrm{Im}}}
\def\Coker{{\mathrm{Coker}}}

\def\rank{{\mathrm{rank\, }}}
\def\Spec{{\mathrm{Spec\; }}}
\def\Proj{{\mathrm{Proj\; }}}

\def\lra{\longrightarrow}
\theoremstyle{plain}
\newtheorem{thm}{Theorem}[section]
\newtheorem{crl}[thm]{Corollary}
\newtheorem{prp}[thm]{Proposition}
\newtheorem{lmm}[thm]{Lemma}
\theoremstyle{definition}
\newtheorem{dfn}[thm]{Definition}

\newtheorem{exa}[thm]{Example}

\newtheorem{rem}[thm]{Remark}
\def\AA{{\mathbb{A}}}
\def\der{{\partial}}
\def\Der{{\mathrm{Der}}}
\usepackage{enumerate}

\def\FF{\mathbb{F}}
\def\I{\mathcal{I}}
\def\tors{\mathrm{tors}}
\def\Bl{\mathrm{Bl}}

\title{F-blowups of normal surface singularities}
\author{Nobuo Hara, Tadakazu Sawada and Takehiko Yasuda}
\address{Mathematical Institute, Tohoku University, Sendai 980-8578, Japan}
\email{hara@math.tohoku.ac.jp}
\thanks{The first author is supported by Grant-in-Aid for Scientific Research (22540039).}
\address{Mathematical Institute, Tohoku University, Sendai 980-8578, Japan}
\email{sa6m17@math.tohoku.ac.jp}
\address{Department of Mathematics and Computer Science,
Kagoshima University,
1-21-35 Korimoto, Kagoshima 890-0065, Japan}
\email{yasuda@sci.kagoshima-u.ac.jp}
\thanks{The third author is supported by Grants-in-Aid for Scientific Research (22740020).}

\begin{document}

\begin{abstract}
We study F-blowups of non-F-regular normal surface singularities. Especially the 
cases of rational double points and simple elliptic singularities are treated in detail.
\end{abstract}

\maketitle

\section{Introduction}

The F-blowup introduced in \cite{Y1} 
is a canonical birational modification of a variety in positive characteristic. 
For a non-negative integer $e$, the $e$-th F-blowup of a variety $X$ is 
defined as the blowup at $F^{e}_{*}\cO_{X}$, that is, the universal birational 
flattening of $F^{e}_{*}\cO_{X}$. Here $F^{e}_{*}\cO_{X}$ is the pushforward 
of the structure sheaf by the $e$-iterated Frobenius morphism. 
It turns out that the F-blowup of a quotient singularity has a connection 
with the $G$-Hilbert scheme \cite{Y1}, \cite{TY}. However, the F-blowup 
has an advantage that it is canonically defined for arbitrary singularity in 
positive characteristic, whreas the $G$-Hilbert scheme is defined only for 
a quotient singularity. Actually, it is proved in \cite{Y1} that the $e$-th 
F-blowup of any curve singularity with $e \gg 0$ is normal, whence resolves 
singularities in dimension one. \par 

As is naturally expected, the F-blowup is also connected to F-singularities 
in positive characteristic such as F-pure and F-regular singularities. It is 
proved that the sequence of F-blowups for an F-pure singularity is monotone 
\cite{Y3} and that the $e$-th F-blowup of an F-regular surface singularity 
coincides with the minimal resolution for $e\gg 0$ \cite{H}. However, it is too 
much to ask for F-blowups of normal surface singularities to be the minimal 
resolution or even smooth in general. Actually, there exist (non-F-regular) 
rational double points whose F-blowups are singular \cite{HS}. \par

Although some good aspects as well as pathologies of F-blowups have 
recently been discovered as above, their behavior is a mystery yet, even 
in dimension two. In this paper, we explore the behavior of F-blowups of 
certain normal surface singularities more in detail. We are mainly concerned 
with two classes of surface singularities, that is, non-F-regular rational 
double points (which exist only in characteristics up to five) 
and simple elliptic singularities. We will discuss F-blowups of these singularities, 
focusing on the normality, smoothness and stabilization of F-blowup sequences. 

For this purpose, we do utilize not only the classical theory of surface singularities,
but also computations with Macaulay2 \cite{M2}, which are complementary to each 
other. The key of our computations is two Macaulay2 functions which we will write 
down. Given a module, the first function computes an ideal such that the blowups 
at the ideal and module coincide, following Villamayor's description of such an ideal 
\cite{Vi}. Using this together with a built-in function  to compute  Rees algebras,
one can explicitly compute a graded ring describing the blowup at a module. The 
second function  which we will write computes the Frobenius pushforward $F_{*}M$
of a given module $M$. These functions enable us to investigate F-blowups in details, 
especially for hypersurface surface singularities in characteristic two or three. 

In the case of rational double points, 
one can apply general theory of rational surface singularities to show that 
F-blowups are normal and dominated by the minimal resolution. Then one 
can determine the $e$-th F-blowup by the direct sum decomposition of 
$F^{e}_{*}M$ into indecomposable modules. We will do this for many types 
of non-F-regular rational double points, in some cases by  theoretical 
arguments and in other cases by computer-aided calculation. In the latter 
cases, computations of the blowups at modules are again useful. For instance, 
one can see with such computation whether two obtained indecomposable 
modules are isomorphic. A particularly interesting obtained result is that for 
$e \ge 2$, the $e$-th F-blowup of $D_{4}^{1}$ and $D_{5}^1$-singularities in characteristic 
two is the minimal resolution, though $D_{4}^{1}$ and $D_{5}^1$-singularities are not F-regular. 
As far as we have computed so far, there is no other non-F-regular rational 
double point such that any of its F-blowups is the minimal resolution. 

We will investigate F-blowups of simple elliptic singularities in detail as well.
Since a simple elliptic singularity $(X,x)$ is quasi-homogeneous in general, 
its minimal resolution $\widetilde X$ has the same structure as the conormal 
bundle over  the elliptic exceptional curve $E$, which is identified with the 
negative section. We compute the torsion-free pullback of $F^e_*\cO_X$ to 
the minimal resolution using its graded structure and the vector bundle 
$F^e_*\cO_E$ over $E$. As a consequence we can determine the structure of 
the F-blowups up to normalization, which turns out to be different according 
to the self-intersection number $E^2$ and whether the singularity $(X,x)$ 
is F-pure or not. We cannot determine whether or not an F-blowup is normal  
in general, but we see that an F-blowup is non-normal in some cases with
Macaulay2 computation. It turns out that 
an F-blowup coincides with the minimal resolution in some cases, but in 
general, F-blowups of simple elliptic singularities behave badly: They are 
non-normal, not dominated by the minimal resolution and the sequence 
of F-blowups does not stabilize.

\section{Preliminaries}

\subsection{Blowups at modules}

Let $X$ be a Noetherian integral scheme and $\M$ a coherent sheaf on $X$.
For a modification $f:Y\to X$,
we denote the torsion-free pullback $(f^{*}\M)/\tors$ by $f^{\star} \M$,
where $\tors$ denotes the subsheaf of torsions.

\begin{dfn}
A modification $f:Y\to X$ is called a \emph{flattening} of $\M$
if $f^{\star} \M$ is flat, or equivalently locally free.
A flattening $f$ is said to be \emph{universal} 
if every flattening  $g : Z \to X$ of $\M$
factors as $g : Z \to Y \xrightarrow {f}X$.
(The universal flattening exists and is unique.
It can be constructed as a subscheme of a Quot scheme. See for instance \cite{OZ,Vi}.)
The universal flattening is also called the \emph{blowup of $X$ at $\M$} and denoted by $\Bl_{\M}(X)$. 
\end{dfn}

The following are basic properties of the blowup at a module, which directly follow from the definition:
\begin{enumerate}
\item The modification $\Bl_{\M}(X) \to X$ is an isomorphism exactly over the locus where $\M$ is flat. 
\item If $\mathcal{N} \subset \M$ is a torsion subsheaf, then $\Bl_{\M}(X) = \Bl_{\M/\mathcal{N}}(X)$.
\item If $\M$ is an ideal sheaf, then the blowup at $\M$ defined above coincides with the usual 
blowup with the center $\M$.
\end{enumerate}

The following are examples of blowups at modules. Therefore one can compute them
in the method explained below.

\begin{exa}
If $X$ is an algebraic variety over a field $k$, then
its Nash blowup is the blowup at $\Omega_{X/k}$, the sheaf of differentials. The higher version of the Nash blowup is also 
an  example of the blowup at  a module (see \cite{Y2}).
\end{exa}

\begin{exa}
Let $Y$ be a quasi-projective algebraic variety, $G$ a finite group of automorphisms of $Y$ and $X:=Y/G$ the quotient
variety. Then the $G$-Hilbert scheme $\Hilb^{G} (Y)$ is defined to be the closure of the set of free $G$-orbits
in the Hilbert scheme of $Y$ (see \cite{IT}).  One can show that $\Hilb^{G} (Y)$ is isomorphic to the blowup at $\pi_{*}\cO_{Y}$,
where $\pi:Y\to X$ is the quotient map.  
\end{exa}

Let $r$ be the rank of $\M$, $K$ the function field of $X$ and
fix an isomorphism $ \bigwedge^{r}\M  \otimes K \cong K $.  
Then define a fractional ideal sheaf
\[
\I_{\M} := \Im(\bigwedge^{r} \M \to \bigwedge^{r} \M  \otimes K \cong K ) . 
\]

\begin{prp}[See \cite{OZ,Vi}]\label{OZ-Vi}
The blowup at $\M$ is isomorphic to the blowup at $\I_{\M}$, 
\[
\Bl_{\I_{\M}}(X)=\Proj_{\!\! X}\!\left ( \bigoplus_{n \ge 0} \I_{\M}^{n} \right).
\]
\end{prp}

Note that although $\I_{\M}$ depends on the choice of the isomorphism 
$ \bigwedge^{r}\M  \otimes K \cong K $, 
the isomorphism class of $\I_{\M}$ and so $\Bl_{\I}(X)$ are independent of it.

We will now recall Villamayor's method \cite{Vi} for computing
  $\I_{\M}$ in the affine case. Suppose that $X =\Spec R$.
Abusing the notation, we identify the sheaf $\M$ with
the corresponding $R$-module $M$, the fractional ideal sheaf $\I_{\M}$
with the fractional ideal $I_{M}\subset K$, and so forth.
Let
\[
 R ^{m} \xrightarrow{A} R^{n} \to M  \to 0
\]
be a presentation of $M$ given by an $n \times m$ matrix $A$.
Here and hereafter we think of elements of  free modules as  column vectors
and the map $R ^{m} \xrightarrow{A} R^{n}$
is given by the left multiplication with $A$, $v \mapsto A v$. We call $A$ a \emph{presentation matrix} of $M$.
Then there exist  $n-r$ columns of $A$ such that if $A'$ denotes the 
submatrix of $A$ formed by these columns, then
\[
M' := \Coker(R^{n-r} \xrightarrow{A'} R^{n})
\]
has rank $r$.  Then $M$ is a quotient of $M'$ by  some  torsion submodule of $M'$. Therefore the blowups at $M$ and $M'$ are equal.

\begin{prp}[\cite{Vi}]\label{prp-Vi}
The ideal generated by $(n-r)$-minors of $A'$, which is by definition the $r$-th Fitting ideal of $M'$,
is equal to $I_{M}$ for a suitable choice of isomorphism 
$\bigwedge^{r} M  \otimes K \cong K$.
\end{prp}

The computation of this ideal is implemented  in Macaulay2 as
\begin{verbatim}
villamayorIdeal = M -> (
  r := rank M;
  P := presentation M;
  s := rank source P;
  t := rank target P;
  I := {}; for j to s-1 when #I < t-r do (J := append(I,j); 
    if rank coker P_J  == t- #J then I = J );
  fittingIdeal(r,coker P_I) 
  );
\end{verbatim}
Once the ideal  $I_{M}$ was computed, then the blowup at $M$ is computed as the projective spectrum of the Rees algebra of the ideal:
\[
\Bl_{M}(X) = \Proj R[I_{M}t], \qquad   R[I_{M}t]:=\bigoplus _{i \ge 0} I_{M}^{i} t^{i} \subset R[t].
\]
The computation of  Rees algebras has been already implemented  in Macaulay2 as \texttt{reesAlgebra}.

The computation of blowups at modules is useful to study modules themselves. 
For instance, one can see that two given modules are not isomorphic 
if the associated blowups are not isomorphic.

%
%
%

\subsection{F-blowups}

Suppose now that $X$ is a Noetherian integral scheme of characteristic 
$p>0$ and that its (absolute) Frobenius morphism $F:X \to X$ is finite. 

\begin{dfn}[\cite{Y1}]
For a non-negative integer $e$, we define the \emph{$e$-th F-blowup} of $X$ to be 
the blowup of $X$ at $F^{e}_{*} \cO_{X}$ and denote it by $\FB_{e}(X)$.
\end{dfn}

\textrm{F}rom \cite{Ku}, if $e>0$, then the flat locus of 
$F^{e}_{*} \cO_{X}$ coincides with the regular locus of $X$.
Therefore the $e$-th F-blowup is an isomorphism exactly over the regular locus.

If $X$ is an algebraic variety over an algebraically closed field $k$, then 
there is a more moduli-theoretic construction of F-blowups, which was actually the original definition
of F-blowups in \cite{Y1}:
The $e$-th F-blowup is isomorphic (over $\Z$) to the closure of the set 
\[
 \{ [(F^{e})^{-1} (x)] \mid  \text{nonsingular point }x \in X(k) \}
\]
in the Hilbert scheme of zero-dimensional subschemes.
Here $(F^{e})^{-1} (x)$ is the scheme-theoretic inverse image and
a closed subscheme of $X$ with length $p^{e \dim X}$, and
$[(F^{e})^{-1} (x)]$ is the corresponding point in the Hilbert scheme.

\subsection{Computing the Frobenius pushforward}

Let us now suppose that $X$ is affine, say $X=\Spec R$.
In order to compute  F-blowups of $X$ along the line explained above,
we need to first compute a presentation of $F^{e}_{*} R$. 
For  later use, 
we will explain more generally how to compute $F^{e}_{*}M$ for 
any finitely generated $R$-module $M$ in the case where $R$ is finitely generated over
the prime field $\FF_{p}$.

\subsubsection{The case of a polynomial ring}
Set $S = \FF_{p}[x_{1}, \dots, x_{n}]$ and $q=p^{e}$.
A monomial $x ^{a} = x_{1}^{a_{1}} \cdots x_{n}^{a_{n}}$ defines an $S$-linear map
\[
 \mu _{x^{a}} : S \to S, \ f \mapsto x^{a}f .
\]
Then we reinterpret this map according to another $S$-module structure on 
$S$ by $g \cdot f := g ^{q}f$.
We denote this new $S$-module by $S'$, which is a free $S$-module of rank $q^{n}$ and nothing but $F^{e}_{*}S$. 
We also denote the  map $\mu _{x^{a}}$  regarded as an endomorphism of $S'$ by $\mu_{x^{a}}'$,
which is nothing but $F^{e}_{*} \mu_{x^{a}}$. 

Let $\Lambda:= \{0,1,\dots,q-1\}^{n}$. Then
 $q^{n}$ monomials $x^{b}$, $b \in \Lambda$ form a standard basis of $S'$.
For such a  monomial $x^{b}$,  we have 
\[
\mu_{x^{a}}(x^{b}) = x^{a+b}= x^{q((a+b)\div q)} x^{(a+b)\% q}.
\]
Here $ \div q$ and $\% q$ respectively denote the quotient and the remainder
by the component-wise division by $q$.
We rewrite it as
\[
 \mu_{x^{a}}'(x^{b}) =x^{(a+b)\div q} \cdot x^{(a+b)\% q}.
\]
Thus we obtain
\begin{lmm} 
The defining matrix,  $U(a,e) = (u_{ij})_{i,j\in \Lambda}$, of $\mu'_{x^{a}}$
with respect to the standard basis
is given by 
\[
u_{ij} = 
\begin{cases}
  x^{(a+j) \div q}   & (i = (a+j) \% q) \\
 0  & (\text{otherwise}) .
\end{cases} 
\]
\end{lmm}

Then for a polynomial $f =\sum_{a} c_{a}x^{a}\in S$, 
if $\mu_{f} :S \to S$ denotes 
the multiplication with $f$, then 
$\mu_{f}'=F^{e}_{*}\mu_{f}$ is defined by the matrix 
\[
U(f,e) := \sum_{a} c_{a} \cdot U(a,e).
\]
Note that since the coefficient field is $\FF_{p}$ and the Frobenius map of $\FF_{p}$
is the identity map, we do not have to change the coefficients $c_{a}$. 

Let 
\[
A = 
\left(
\begin{array}{ccc}a_{11} & \cdots & a_{1m} \\ 
\vdots & \ddots & \vdots \\
a_{l1} & \cdots & a_{lm}\end{array}
\right),
\]
be a $l \times m$ matrix with entries in $S$, which defines
an $S$-linear map  $S^{m} \to S^{l}$ denoted again by $A$.  
Then the $F_{*}^{e}A: (S')^{\oplus m} \to (S')^{\oplus l}$ is given by $q^{n}l \times q^{n}m$ matrix
\[
U(A,e) = 
\left(
\begin{array}{ccc}U(a_{11},e) & \cdots & U(a_{1m},e) \\ 
\vdots & \ddots & \vdots \\
U(a_{l1},e) & \cdots & U(a_{lm},e)\end{array}
\right).
\]
As a consequence, we have
\begin{prp}
If $A$ is a presentation matrix of an $S$-module $M$,
then $U(A,e)$ is a presentation matrix of $F^{e}_{*}M$. 
\end{prp}

\subsubsection{The general case}

Suppose that $R$ is the quotient ring $S /I$, $I=(f_{1},\dots,f_{l})$, 
and $M$ is a finitely generated $R$-module. 
Then we first have to compute a presentation of $M$ as an $S$-module.
Let $A$ be a matrix with entries in $S$ and let $\bar A$ be the matrix with entries in $R$ induced from $A$.
Suppose that $\bar A$ is a presentation matrix of $M$:
\[
R^{m} \xrightarrow{\bar A} R^{n}\to M \to 0 .
\]
Let $\tilde M$ be the $S$-module with the presentation matrix $A$: 
\[
S^{m} \xrightarrow{A} S^{n}\to \tilde M \to 0 .
\]
Then $M = R \otimes _{S} \tilde M$.
The  $S$-module $R$ has a standard presentation
\[
S^{l} \xrightarrow{(f_{1} , \dots, f_{l})} S \to R  \to 0.
\]
Now a presentation of $M$ as an $S$-module can be computed from those of $R$ and $\tilde M$. 

If $B$ is a presentation matrix of $M$ as an $S$-module, 
then $U(B,e)$ is one of $F^{e}_{*}M$ as an $S$-module. If $\overline{U(B,e)} $ denotes the matrix with entries in $R$
induced from $U(B,e)$, then $\overline{U(B,e)} $ is a presentation matrix of $F^{e}_{*}M$ as an $R$-module.

\subsubsection{Implementation to Macaulay2}

The following Macaulay2 function returns the pushforward $F^{e}_{*}M$ of the given module $M$,
following the recipe explained above:
\begin{verbatim}
frobeniusPushForward = (M,e) -> ( 
  R := ring M;
  p := char R;
  assert(p>0);
  q := p^e; 
  I := ideal R;
  l := numgens I;
  B := gens ideal R;
  S := ambient R;                            
  n := numgens S;  
  qSequence := i -> apply(0..n-1, j -> (i % q^(n-j)) // q^(n-j-1));
  toNumber := i -> sum(n , j ->  i_j * q ^ (n-j-1)  );
  qQuotient := i -> apply(i, j -> j // q);
  qRemainder := i -> apply(i, j -> j % q);
  monoToMatrix := m ->  
    (coefficients m)_1_(0,0) 
      * map(S^(q^n),S^(q^n),   
          (i,j) -> (e =( toList qSequence i ) + (exponents m)_0;  
                     if (toNumber qRemainder e) == j 
    	               then S_(toList qQuotient e) 
                       else 0));
  polyToMatrix := f -> 
    if f == 0_S 
      then map(S^(q^n),S^(q^n),0_S) 
      else sum(terms f, i -> monoToMatrix i);
  basisToMatrix := b -> 
    fold((i,j)->(i|j), apply(( flatten entries b), polyToMatrix ) );
  matrixToMatrix := m -> 
    fold((i,j)->(i||j), 
      apply(apply(entries m, i -> matrix{i}), basisToMatrix));      
  ROverS := coker map(S^1,S^l, entries B); 
  PresenOverR := presentation minimalPresentation M;
  PresenOverS := presentation minimalPresentation (
                                coker(sub(PresenOverR,S))** ROverS);   
  L := matrixToMatrix PresenOverS;
  minimalPresentation coker sub(L,R)
  );     
\end{verbatim}
Note that in the computation with Macaulay2, columns and rows of matrices should be
indexed by single indices rather than multi indices. 
For this aim, the above inner functions \texttt{qSequence} and \texttt{toNumber} define bijections
between the sets $\{0, 1,\dots,q^{n}-1\}$ and $\Lambda$ which are inverses to each other.

We should also note that one can compute $F_{*}^{e}M$ 
also with the built-in function \texttt{PushForward}
in the case where the ring and the module are (weighted) homogeneous.

\subsection{Computing the singular and non-normal loci of a blowup}

We often would like to know if a given blowup is smooth or normal,
or to know where the singular locus or the non-normal locus is. 
One way to compute the singular locus of $\Bl_{I}(X)$ is 
to compute the singular locus of $\Spec R[It]$.
For instance, suppose that we have an expression of  $R[It]$ as
a quotient of a polynomial ring over $R$, 
\[
R[It]=R[t_{1},\dots,t_{n}] / J .
\]
Then
$\Bl_{I}(X)$ is smooth if and only if the singular locus of $\Spec R[It]$ is contained
in the closed subset $V(t_{1},\dots,t_{n})\subset \Spec R[It]$. 
This method is useful when the Rees algebra is relatively simple.
Otherwise, the computation  may not finish in a reasonable time.

In that case, 
an alternative way is to compute the singular loci of affine charts.
With the above notation, the blowup $\Bl_{I}(X)$ is covered by $n$ affine charts corresponding
to the variables $t_{1} , \dots, t_{n}$. Their coordinate rings are
\[
R[t_{1},\dots,t_{n}] / (J  + (t_{i}-1) ) , \ i=1,\dots,n.
\]
These rings are likely to become simpler than  $R[It]$
and easier to compute the singular loci.
Computation of these rings is implemented as follows: 
\begin{verbatim}
affineCharts = S ->( 
      T := (flattenRing S)_0;
      varsOfS := apply(flatten entries vars S,i->sub(i,T));
      apply(varsOfS, i-> minimalPresentation(T/ideal(i - 1 )))) ;
\end{verbatim}

The same method can apply to find the non-normal locus.

\subsection{Embedding F-blowups into the Grassmannian and the projective space}

As already mentioned above, F-blowups are constructed as a subscheme 
of the Grassmannian. Then further composing with the Pl\"ucker embedding, we obtain an embedding into a projective space over $X$. 

To describe this embedding, let $X=\Spec R$ be of dimension $n$, 
$K$ the function field of $X$, and let the fractional ideal 
$I=\Im(\bigwedge^{p^n} R^{1/p^e} \to K)$ be generated by $m+1$ 
elements $s_0,\dots,s_m$. Then, being the blowup of $X$ at $I$, 
the $e$-th F-blowup $\FB_e(X)$ of $X$ is embedded into the projective 
space $\PP^m_X$ over $X$. 

Suppose now that $f \colon Y \to X$ is any flattening of 
$R^{1/p^e} \cong F^e_*\cO_X$. Then we have a surjection 
$\cO_Y^{\oplus m+1} \to f^{\star}I=\det f^{\star}R^{1/p^e}$
induced by $s_0,\dots,s_m$, which gives rise to a morphism 
$\Phi_e \colon Y \to \PP^m_X$ 
such that $\Phi_e^*\cO_{\PP}(1) \cong \det f^{\star}R^{1/p^e}$, 
and the image $\Phi_e(Y)$ of this morphism is nothing but $\FB_e(X)=\Bl_I(X)$. 

In dimension two where the existence of resolution of singularities 
is established in arbitrary characteristic, we can study F-blowups 
downwards from a resolution that flattens the $\cO_X$-module 
$F^e_*\cO_X \cong \cO_X^{1/p^e}$. The following is an immediate consequence of the above observation. 

\begin{prp}\label{FBsurface}
Let $X$ be a surface over $k$ and let $f \colon \widetilde{X} \to X$ 
be a resolution with irreducible exceptional curves $E_1,\dots,E_s$. 
Suppose that $f^{\star}\cO_X^{1/p^e}$ is flat, so that we have a 
birational morphism $\Phi_e \colon \widetilde{X} \to \FB_e(X)$. 
Then 
$\Phi_e(E_i)$ is a curve on $\FB_e(X)$ if $c_1(f^{\star}\cO_X^{1/p^e})E_i>0$, 
and 
$E_i$ contracts to a point on $\FB_e(X)$ if $c_1(f^{\star}\cO_X^{1/p^e})E_i=0$.
\end{prp}

\section{F-blowups of rational surface singularities}

Throughout this section we work under the following notation: 

\medskip
$k$: an algebraically closed field of characteristic $p>0$ \par
$(X,x)$: a rational surface singularity defined over $k$ 
         with local ring $R=\cO_{X,x}$                      \par
$f\colon \widetilde{X} \to X$: the minimal resolution of $(X,x)$ 
                               with Exc$(f)=\bigcup_{i=1}^s E_i$ 

\medskip
The situation is quite simple in this case, because of the following 
fact \cite{AV}: If $M$ is a reflexive $\cO_X$-module\footnote{
We always assume that $M$ is a finitely generated $\cO_X$-module.}, 
then its torsion-free pullback $\widetilde{M}=f^{\star}M=f^*M$/torsion 
is an $f$-generated locally free $\cO_{\widetilde X}$-module such 
that $f_*\widetilde{M}=M$ and $R^1f_*\widetilde{M}=0$. Note that 
this vanishing of higher direct image is an easy consequence of 
the rationality of the singularity $(X,x)$ and the $f$-generation 
of $\widetilde M$, which gives rise to a surjection 
$\cO_{\widetilde X}^{\oplus n} \to\hspace{-10pt}\to \widetilde{M}$. 

We also have 

\begin{lmm}[{\cite[Lemma 1.8]{H}}]\label{complete}
If $M$ is a reflexive $\cO_X$-module of rank $r$, then the natural map 
$$\bigwedge^r M \to f_*(\det\widetilde{M})$$
is surjective. 
\end{lmm}

\begin{prp}
The $e$-th F-blowup $\FB_e(X)$ of a rational surface singularity $(X,x)$ 
is dominated by the minimal resolution $\widetilde X$ and has only 
rational singularities for all $e\ge 0$. 
\end{prp}

\begin{proof}
Since $M:=R^{1/p^e}$ is a reflexive $R$-module, its torsion-free 
pullback $\widetilde{M}=f^{\star}R^{1/p^e}$ to $\widetilde X$ is 
flat, so that the minimal resolution $f\colon\widetilde{X}\to X$ 
factors through the universal flattening $\FB_e(X)$ of $R^{1/p^e}$. 
On the other hand, the ideal $I=I_M$ for $M=R^{1/p^e}$ is 
$$I=H^0(\widetilde{X},\det\widetilde{M})$$ 
by Lemma \ref{complete}, so that we can take $I$ to be an integrally 
closed ideal in $R$, 
or \emph{complete} ideal in the sense of Lipman \cite{L}. Then the 
Rees algebra $R[It]$ is normal by \cite[Proposition 8.1]{L}, so 
$\FB_e(X)=\Proj R[It]$ is normal. It then follows from \cite{A1} 
that $\FB_e(X)$ has only rational singularities. 
\end{proof}

\begin{crl}
Let $(X,x)$ be a rational surface singularity over $k$. 
\begin{enumerate}
\renewcommand{\labelenumi}{(\arabic{enumi})}
\item For any $e\ge 0$, the $e$-th F-blowup $\FB_e(X)$ is obtained 
by contracting part of exceptional curves $E_1,\dots,E_s$ on the 
minimal resolution $\widetilde X$ to normal points with at most 
rational singularities. 
\item The  minimal resolution $\widetilde X$ of $(X,x)$ is obtained 
by finitely many iteration of F-blowups. More explicitly, for any 
sequence of positive integers $e_1,\dots,e_s$, we have 
$\widetilde{X}=
 \FB_{e_s}(\FB_{e_{s-1}}(\cdots\FB_{e_2}(\FB_{e_1}(X))\cdots))$. 
\end{enumerate}
\end{crl}

The behavior of F-blowups is especially nice for F-regular surface 
singularities. Namely, the $e$-th F-blowup of any F-regular surface 
singularity is the minimal resolution for $e\gg 0$ \cite{H}. 
We next consider F-blowups of non-F-regular rational double points 
more in detail. 


\medskip 
In what follows, we use the notation of Artin in \cite{A2} for rational 
double points in positive characteristic. 

Among non-F-regular rational double points, the case of Frobenius sandwiches is easier 
to analyze their F-blowups. Let $X$ be a Frobenius sandwich of a smooth surface $S$, 
i.e., the Frobenius morphism of $S$ factors as $F: S \xrightarrow{\pi} X \rightarrow S$. 
Then F-blowups of $X$ are also the universal flattening of the reflexive $\cO_X$-module 
$\pi_{\ast} \cO_{S}$ (\cite[Proposition 4.3]{HS}). Thanks to this observation, we can study 
F-blowups of the Frobenius sandwich $X$ via 
$\pi_*\cO_S$ instead of $F^e_*\cO_X$. 
For example, we find whether the irreducible exceptional curve $E_i$ appears on 
$\FB_e (X)$ or not by evaluating the intersection number $c_1(f^{\star}(\pi_*\cO_S))E_i$ 
in Proposition \ref{FBsurface}. 

\subsection{$D_{2n}^0$-singularities.} 
Here we consider a $D_{2n}^0$-singularity for $n \geq 2$ in $p=2$ as a Frobenius sandwich. 
Let $\AA^2=\Spec k[x,y]$ and $\pi \colon \AA^2 \to X=\AA^2/\delta$ the quotient 
map by a vector field 
$\delta=(x^2+nxy^{n-1}){\der}/{\der x}+y^n {\der}/{\der y} \in \Der_k\, \cO_{\AA^2}$. 
Here 
$$\cO_X=k[x,y]^{\delta}=k[x^2,y^2,x^2y+xy^n] \cong k[X,Y,Z]/(Z^2+X^2Y+XY^n)$$
and $X$ has a $D_{2n}^0$-singularity. Then $R=\cO_X=k[X,Y,Z]/(Z^2+X^2Y+XY^n)$ is a 
graded ring with $\deg X = 2(n-1), \deg Y = 2$ and $\deg Z = 2n-1$. The $4(n-1)$-th 
Veronese ring of $R$ is 
$$R^{(4(n-1))}=k[X^2,Y^{2(n-1)},XY^{n-1}] \cong k[u,v,w]/(w^2-uv).$$
Set $x_0 = u^{1/2} = X = x^2$ and $x_1 = v^{1/2} = Y^{n-1} = y^{2(n-1)}$. Then 
$R^{(4(n-1))}\cong k[x_0^2,x_1^2,x_0x_1] = k[x_0,x_1]^{(2)}$, so that $\Proj R\cong\PP^1$ 
with homogeneous coordinates $(x_0:x_1)=(x^2:y^{2(n-1)})$. Let $s=x_1/x_0=y^{2(n-1)}/x^2$ 
be the affine coordinate of $U_0=D_+(x_0) \subset \Proj R \cong \PP^1$ and pick a 
homogeneous element $t=Z/X=y(x+y^{n-1})/x \in R$ of degree $1$. 
Since 
$$t^{2(n-1)}=\dfrac{x_1(x_1-x_0)^{n-1}}{x_0^{n-1}},$$
the $\Q$-divisor 
$$D=\dfrac{1}{2(n-1)}(0)+\dfrac{1}{2}(1)-\dfrac{1}{2}(\infty)$$
on $\PP^1$ gives $R=\bigoplus_{n \geq 0} H^0(\PP^1,nD)t^n$ (the Pinkham-Demazure construction). 

Let $g\colon X' \rightarrow X=\Spec R$ be the weighted blowup with respect 
to the weight $(2(n-1), 2, 2n-1)$. Then 
$X' \cong \Spec_{\!\!\PP^1}(\bigoplus_{n \geq 0}\cO_{\PP^1}(nD)t^n)$ admits an  
affine morphism $\rho \colon X' \to \PP^1$ that is an $\AA^1$-bundle over 
$\PP^1 \backslash \{0,1,\infty \}$, and the exceptional curve of $g$ is the 
negative section $E\cong \PP^1$ of $\rho$. Let $X_0'=\rho^{-1}U_0$. 
Then 
$$\cO_{X_0'}=k\left[ s,t, \dfrac{t^2}{s-1}, \dfrac{t^3}{s-1}, \ldots , 
\dfrac{t^{2(n-1)-2}}{(s-1)^{n-2}},\dfrac{t^{2(n-1)-1}}{(s-1)^{n-2}},\dfrac{t^{2(n-1)}}{s(s-1)^{n-1}}\right]
$$
and $X'$ has an $A_{2n-3}$-singularity on $E|_{X_0'} \cong \Spec k[s]$ at $s=0$. 

To resolve the $A_{2n-3}$-singularity, we may replace $X_0'=\rho^{-1}U_0$ 
by $V=\rho^{-1}(U_0 \backslash \{1\})$. The affine coordinate ring of $V$ is 
$$\cO_V = \cO_{X_0'}\left[\dfrac{1}{s-1}\right]=k[s,t,t^{2(n-1)}/s]_{s-1}.$$
The minimal resolution $h \colon \widetilde{V} \to V$ of $V$ is given by 
$\widetilde{V}=\bigcup_{i=1}^{2(n-1)} \widetilde{V}_i$ 
where $\widetilde{V_i}=\Spec k[s/t^{i-1},t^i/s]_{s-1}$. 
Let $\widetilde{E} \cong \PP^1$ be the $h$-exceptional curve lying on 
$\widetilde{V}_{n-2} \cup \widetilde{V}_{n-1}$: 
$$
\widetilde{E} = \Spec k[t^{n-2}/s] \cup \Spec k[s/t^{n-2}] 
              \subset \widetilde{V}_{n-2} \cup \widetilde{V}_{n-1}.$$

Now suppose that $n$ is even; $n=2k$ with $k \geq 1$. 
Let $f=t^{n-2}/s(s-1)^{k-1}$ and $g=t^{n-1}/s(s-1)^{k-1}$. 
Then $f,g \in \cO_{\widetilde{V}_{n-2}}$ and $x=g+yf$. 
Thus
$$(h \circ g)^{\star} (\pi_{\ast} \cO_{\AA^2})|_{\widetilde{V}_{n-2}}
   = \Im(\cO_{\widetilde{V}_{n-2}}\otimes_{\cO_X}\cO_{\AA^2}\rightarrow k(\AA^2))
   = k[s/t^{n-3},t^{n-2}/s,x,y]_{s-1}$$
is a free $\cO_{\widetilde{V}_{n-2}}$-module with basis $1,y$. 
Similarly it follows that $(h\circ g)^{\star}(\pi_*\cO_{\AA^2})|_{\widetilde{V}_{n-1}}$ 
is a free $\cO_{\widetilde{V}_{n-1}}$-module with basis $1,x$. 
The transition matrix of the two bases on $\widetilde{V}_{n-2} \cap \widetilde{V}_{n-1}$ 
is given by  
$$(1\ x)=(1\ y)\left( \begin{array}{cc}
                     1 & t^{n-1}/s(s-1)^{k-1} \\
                     0 & t^{n-2}/s(s-1)^{k-1}
                      \end{array}\right).$$
Since $s-1$ is a unit on $V$, the intersection number of 
$L=c_1((h\circ g)^{\star}(\pi_*\cO_{\AA^2}))$ with $\widetilde{E}$ is $L\widetilde{E}=1$. 
By the following lemma in \cite{AV}, this means that the reflexive $\cO_X$-module 
$\pi_{\ast}\cO_{\AA^2}$ of rank $2$ is the indecomposable one corresponding to 
$\widetilde{E}$, or equivalently, to the solid circle in the figure below, on 
the minimal resolution $\widetilde{X}$: 

\noindent {\unitlength 0.1in
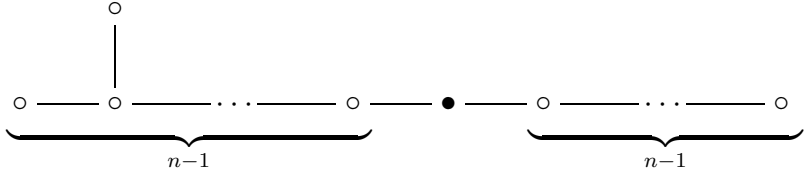
\begin{picture}( 60.0000, 11.5000)
\put(2.0000,3.5000){\makebox(0,0)[lb]{$\xygraph{
    \circ  - [r]
    \circ (
        - [u] \circ ,
        - [r]
)} \cdots \xygraph{
           - [r] \circ
           - [r] \bullet 
           - [r] \circ 
           - [r]
)} \cdots \xygraph{
           - [r] \circ
)}$}}
\put(2.200,0.8000){\makebox(0,0)[lb]{$\underbrace{\,\hspace{47mm}\,}_{n-1}$}}
\put(29.5000,0.8000){\makebox(0,0)[lb]{$\underbrace{\,\hspace{35mm}\,}_{n-1}$}}
\end{picture}}\vspace{1mm}
%
In the case where $n=2k+1$ with $k \geq 1$, we obtain the same conclusion that 
$c_1((h\circ g)^{\star} (\pi_{\ast} \cO_{\AA^2})) \cdot \widetilde{E} =1$ and 
$c_1((h\circ g)^{\star} (\pi_{\ast} \cO_{\AA^2})) \cdot E_i =0$ if $E_i\ne\widetilde{E}$.

Therefore $\FB_e (X)$ coincides with the normal surface obtained by 
contracting all exceptional curves on $\widetilde X$ except $\widetilde{E}$,
for all $e \geq 1$. 
In particular, $\FB_e (X)$ is not the minimal resolution for all $e \geq 1$. 

Similarly we see that the F-blowups 
of rational double points of type $E_7^0$, $E_8^0$ in $p=2$, 
$E_6^0$, $E_8^0$ in $p=3$ and $E_8^0$ in $p=5$ 
are obtained by contracting exceptional curves on the minimal resolution except 
one irreducible exceptional curve; 
see \cite{HS}. 

\begin{lmm}[Artin-Verdier \cite{AV}]\label{McKay}
Let $(X,x)$ be a two-dimensional rational double point and let $f : \widetilde{X} \rightarrow X$ be the minimal resolution. 
Let $E_1, \ldots , E_s$ be the irreducible exceptional curves of $f$. 
Then there exists a one-to-one correspondence between the exceptional curves $E_i$ of $f$ 
and the isomorphism classes of non-trivial indecomposable reflexive $\cO_X$-modules $M_i$, 
and $c_1(f^{\star}M_i) \cdot E_j=\delta_{ij}$. 
\end{lmm}

\subsection{} We can study F-blowups of other rational double points 
with Macaulay2 computation. 
First we consider a $D_4^1$-singularity in $p=2$. 
Let $R=k[x,y,z]/(z^2+x^2y+xy^2+xyz)$ and $X=\Spec R$. 
Then $X$ has a $D_4^1$-singularity. 
Using Macaulay2 function \verb+frobeniusPushForward+ in Section 2.3.3, 
we see that the presentation matrix of 
$F_{\ast} R$ is equivalent to 
$$\left(\begin{array}{cc}
z  & x+y+z  \\
xy & z    
\end{array}\right) \oplus 
\left(\begin{array}{cc}
z        & y  \\
x(x+y+z) & z    
\end{array}\right) \oplus 
\left(\begin{array}{cc}
z & y(x+y+z)  \\
x & z    
\end{array}\right) \oplus 0,$$
where $0$ is the zero matrix of size $1$.
Then the cokernel of each matrix of size $2$ defines 
a nontrivial reflexive 
$R$-module of rank $1$ and those reflexive 
$R$-modules are different from each other. Thus $\FB_1 (X)$ coincides 
with the normal surface obtained by contracting the exceptional curve 
on the minimal resolution $\widetilde{X}$ corresponding to the blank circle in the figure below: 
$$\xygraph{
    \bullet  - [r]
    \circ (
        - [u] \bullet ,
        - [r] \bullet
)}$$
Furthermore, we see that the reflexive 
$R$-module corresponding to the 
central curve in the above figure appears as a direct summand of the Frobenius pushforward of each nontrivial rank $1$ reflexive module. 
Thus $\FB_e (X)$ is the minimal resolution for $e \geq 2$, since the 
$D_4^1$-singularity is F-pure. 
A similar result holds for the case of a $D_5^1$-singularity. 
Note that $D_4^1$ and $D_5^1$-singularities are not F-regular. 

\begin{rem}
The $D_4^1$-singularity is a (wild) quotient singularity, i.e., 
there exists a finite group $G$ of automorphisms of $M=\Spec k[[x,y]]$ 
such that the quotient $X:=M/G$ has the $D_4^1$-singularity. 
F-blowups of a tame quotient singularity are always dominated by the $G$-Hilbert scheme, 
but the same dose not hold for the $D_4^1$-singularity. 
Let $R=k[[x,y]]^G \subset S=k[[x,y]]$ be the invariant subring. 
Then $S$ is an $R$-module of rank $2$. 
Thus the blowup of $X$ at $R$-module $S$, which coincides with the $G$-Hilbert scheme $\Hilb^{G} (M)$, 
has at most two irreducible exceptional curves. 
On the other hand, F-blowups $\FB_e (X)$ have three or four irreducible exceptional curves as we have seen above. 
Hence F-blowups $\FB_e (X)$ of the $D_4^1$-singularity are not dominated by the $G$-Hilbert scheme $\Hilb^{G} (M)$ for all $e \geq 1$. 
\end{rem}

Next we consider an $E_6^0$-singularity in $p=2$. 
Let $R=k[x,y,z]/(z^2+x^3+y^2z)$ and $X=\Spec R$. 
Then $X$ has an $E_6^0$-singularity. Write 
$$A_1 = \left(\begin{array}{cccc}
z   & y   & x  & 0 \\
yz  & z   & 0  & x \\
x^2 & 0   & z  & y \\
0   & x^2 & yz & z
\end{array}\right)\hspace{-1mm} ,\hspace{3mm} A_2 = \left(\begin{array}{cccc}
x & y^2+z & y & 0     \\
z & x^2   & 0 & xy    \\
0 & 0     & x & y^2+z \\
0 & 0     & z & x^2
\end{array}\right)$$
and $A_3 = {}^tA_2$. 
Then their cokernels define nontrivial reflexive 
$R$-modules of rank $2$ 
and those 
$R$-modules are different from each other. Now we see that 
presentation matrices of 
$F_{\ast} R$ and $F^2_{\ast} R$ are equivalent to 
$A_1^{\oplus 2}$ and $A_1^{\oplus 4} \oplus A_2^{\oplus 2} \oplus A_3^{\oplus 2}$, 
respectively. Furthermore, a direct summand other than $A_1, A_2$ and $A_3$ 
does not appear in the presentation matrices of 
$F^e_{\ast} R$ for $e \geq 2$. 
Since the blowup of $X$ at $\Coker\, A_1$ has the only one singular point, 
$\FB_e (X)$ coincides with the normal surface obtained by contracting the 
exceptional curves on the minimal resolution $\widetilde{X}$ corresponding 
to the blank circles in the figure below: 

\medskip \medskip \medskip \noindent$e=1$:
\raisebox{-5mm}{$\xygraph{
    \circ - [r]
    \circ - [r]
    \circ (
        - [u] \bullet ,
        - [r] \circ
        - [r] \circ
)}$}
\quad \quad $e \geq 2$:
\raisebox{-5mm}{$\xygraph{
    \circ   - [r]
    \bullet - [r]
    \circ (
        - [u] \bullet ,
        - [r] \bullet
        - [r] \circ
)}$}

\medskip \medskip \medskip The results so far are summarized as follows: 

\begin{prp}[cf. {\cite[Example 4.8]{HS}}]\label{FB_of_RDP}
Let $(X,x)$ be a rational double point of type 
$D_{2n}^0$ for $n \geq 2$, $D_4^1$, $D_5^1$, $E_6^0$, $E_7^0$, $E_8^0$ in $p=2$, 
$E_6^0$, $E_8^0$ in $p=3$ or $E_8^0$ in $p=5$. 
Then F-blowups $\FB_e (X)$ coincide 
with the normal surfaces obtained by contracting the exceptional curves 
on the minimal resolution $\widetilde{X}$ 
corresponding to the blank circles in the figure below{\rm :}

\medskip \noindent {\rm (1)} $D_{2n}^0$-singularity for $n \geq 2$ in $p=2${\rm :} 

\noindent {\unitlength 0.1in
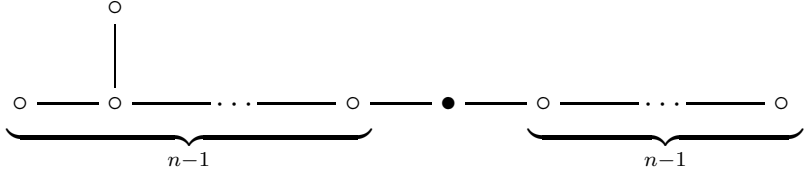
\begin{picture}( 60.0000, 11.5000)
\put(2.0000,3.5000){\makebox(0,0)[lb]{$\xygraph{
    \circ  - [r]
    \circ (
        - [u] \circ ,
        - [r]
)} \cdots \xygraph{
           - [r] \circ
           - [r] \bullet 
           - [r] \circ 
           - [r]
)} \cdots \xygraph{
           - [r] \circ
)}$}}
\put(2.200,0.8000){\makebox(0,0)[lb]{$\underbrace{\,\hspace{47mm}\,}_{n-1}$}}
\put(29.5000,0.8000){\makebox(0,0)[lb]{$\underbrace{\,\hspace{35mm}\,}_{n-1}$}}
\end{picture}}

\noindent {\rm (2)} $D_4^1\ (resp.\ D_5^1)$-singularity in $p=2${\rm :}

\medskip \medskip \noindent $e=1${\rm :} \quad \raisebox{-5mm}{$\xygraph{
    \bullet  - [r]
    \circ (
        - [u] \bullet ,
        - [r] \bullet
)}$} \quad \quad $\left( resp. \hspace{3mm}
\raisebox{-5mm}{\xygraph{
    \circ  - [r]
    \bullet (
        - [u] \circ ,
        - [r] \circ
        - [r] \bullet
)}} \right)$

\medskip \medskip On the other hand, $F$-blowups $\FB_e (X)$ of both singularities are the minimal resolution for $e \geq 2$. 

\medskip \noindent {\rm (3)} $E_6^0$-singularity in $p=2${\rm :}

\medskip \noindent $e=1${\rm :}
\raisebox{-5mm}{$\xygraph{
    \circ - [r]
    \circ - [r]
    \circ (
        - [u] \bullet ,
        - [r] \circ
        - [r] \circ
)}$}
\quad \quad $e \geq 2${\rm :}
\raisebox{-5mm}{$\xygraph{
    \circ   - [r]
    \bullet - [r]
    \circ (
        - [u] \bullet ,
        - [r] \bullet
        - [r] \circ
)}$}

\medskip \medskip \medskip \noindent {\rm (4)} $E_7^0$-singularity in $p=2${\rm :}\quad 
\raisebox{-5mm}{$\xygraph{
    \circ  - [r]
    \circ  - [r]
    \circ (
        - [u] \bullet ,
        - [r] \circ
        - [r] \circ
        - [r] \circ
)}$}

\medskip \medskip \medskip \noindent {\rm (5)} $E_8^0$-singularity in $p=2${\rm :}\quad
\raisebox{-6mm}{$\xygraph{
    \bullet  - [r]
    \circ  - [r]
    \circ (
        - [u] \circ ,
        - [r] \circ
        - [r] \circ
        - [r] \circ
        - [r] \circ
)}$}

\medskip \medskip \medskip \noindent {\rm (6)} $E_6^0$-singularity in $p=3${\rm :}\quad
\raisebox{-5mm}{$\xygraph{
    \circ  - [r]
    \circ  - [r]
    \bullet (
        - [u] \circ ,
        - [r] \circ
        - [r] \circ
)}$}

\medskip \medskip \medskip \noindent {\rm (7)} $E_8^0$-singularity in $p=3${\rm :}\quad
\raisebox{-5mm}{$\xygraph{
    \circ  - [r]
    \circ  - [r]
    \circ (
        - [u] \bullet ,
        - [r] \circ
        - [r] \circ
        - [r] \circ
        - [r] \circ
)}$}

\medskip \medskip \medskip \noindent {\rm (8)} $E_8^0$-singularity in $p=5${\rm :}\quad
\raisebox{-5mm}{$\xygraph{
    \circ  - [r]
    \circ  - [r]
    \circ (
        - [u] \circ ,
        - [r] \bullet
        - [r] \circ
        - [r] \circ
        - [r] \circ
)}$}
\end{prp}

\medskip \medskip For other rational double points, we have the following:

\begin{exa}
(1) $E_6^1$-singularity in $p=2$: 
Let $R=k[x,y,z]/(z^2+x^2y+xy^2+xyz)$ and $X=\Spec R$. 
Then $X$ has an $E_6^1$-singularity. Write 
$$A = \left(\begin{array}{cccccc}
z      & 0   & 0     & 0      & x & z  \\
0      & z   & y     & 0      & y & x  \\
xy     & yz  & z     & x^2+yz & 0 & 0  \\
0      & 0   & x     & x      & y & 0  \\
x^2    & xz  & 0     & yz     & z & 0  \\
xy+y^2 & x^2 & 0     & xy     & 0 & z
\end{array}\right).$$
Then the cokernel of $A$ defines an indecomposable reflexive 
$R$-module of 
rank $3$. The presentation matrix of 
$F_{\ast} R$ is equivalent to $A \oplus 0$, where 
$0$ is the zero matrix of size $1$. Thus $\FB_1 (X)$ coincides with the normal surface 
obtained by contracting the exceptional curves on the minimal resolution $\widetilde{X}$ 
corresponding to the blank circles in the figure below: 
$$\xygraph{
    \circ  - [r]
    \circ  - [r]
    \bullet (
        - [u] \circ ,
        - [r] \circ
        - [r] \circ
)}$$
\medskip \noindent (2) $E_8^3$-singularity in $p=2$: 
Let $R=k[x,y,z]/(z^2+x^3+y^5+y^3z)$ and $X=\Spec R$. 
Then $X$ has an $E_8^3$-singularity. 
In this case, 
$F_{\ast} R$ has two kinds of indecomposable reflexive 
$R$-modules. 
Since 
$\rank F_{\ast} R = 4$, we see that 
$F_{\ast} R$ is a direct sum of 
indecomposable reflexive 
$R$-modules of rank $2$ corresponding to 
the solid circles in the figure below: 
$$\xygraph{
    \bullet  - [r]
    \circ    - [r]
    \circ (
        - [u] \circ ,
        - [r] \circ
        - [r] \circ
        - [r] \circ
        - [r] \bullet
)}$$
Thus $\FB_1 (X)$ coincides with the normal surface obtained by contracting the 
exceptional curves on the minimal resolution $\widetilde{X}$ corresponding to the 
blank circles in the above figure. 
\end{exa}

\section{F-blowups of simple elliptic singularities}

In this section $(X,x)$ will denote a simple elliptic singularity defined 
over an algebraically closed field $k$ of characteristic $p>0$ unless otherwise 
noted. Then by a result of Hirokado \cite{Hi}, $(X,x)$ is quasi-homogeneous. 
So we may assume that $X=\Spec R$ for a graded $k$-algebra 
$$R=R(E,L)=\bigoplus_{n\ge 0} H^0(E,L^n)t^n,$$ 
where $E$ is an elliptic curve over $k$, $L$ is an ample line bundle on 
$E$ and $\deg t=1$. The minimal resolution $f \colon \widetilde{X} \to X$ 
of $X$ is described as follows: $\widetilde X$ has an $\mathbb{A}^1$-bundle 
structure $\pi\colon \widetilde{X}=\mathrm{Spec}_E(L^nt^n) \to E$ over $E$, 
and its zero-section, which we also denote by $E$, is the exceptional curve 
of $f$. Its self-intersection number is $E^2=-\deg L$. Our situation is 
summarized in the following diagram: 
$$\xymatrix
{E\ \ar@<-0.8mm>@{^{(}->}[r] \ar[dr]_{\mathrm{id}} 
 & \widetilde{X} \ar[r]^f\ar[d]^{\pi} & X      \\ 
 &      E                             &  }
$$
To compute the F-blowup $\FB_e(X)$ of $X$, we will look at the structure 
of the torsion-free pullback $f^{\star}R^{1/q}$ of $R^{1/q}\cong F_*^e\cO_X$, 
where $q=p^e$. For this purpose we decompose 
$R^{1/q}=\bigoplus_{n\ge 0} H^0(E,F^e_*L^n)t^{n/q}$ as 
$R^{1/q} = \bigoplus_{i=0}^{q-1}\, [R^{1/q}]_{i/q\!\!\mod\Z}$, where 
$$[R^{1/q}]_{i/q\!\!\!\mod\Z}
  =\bigoplus_{0\le n\equiv i\!\!\!\mod q} H^0(E,F^e_*L^n)t^{n/q}
  \cong \bigoplus_{m\ge 0} H^0(E,L^m\otimes F^e_*L^i)$$ 
is an $R$-summand of $R^{1/q}$ for $i=0,1,\dots ,q-1$; cf.\ \cite{SVdB}. 

In what follows we put $q=p^e$ and $d=\deg L=-E^2$. 

\begin{lmm}\label{lem1}
If $1\le i\le q-1$ and $q\ne di$, then 
$\widetilde X$ is a flattening of $[R^{1/q}]_{i/q\!\!\mod\Z}$. 
\end{lmm}

\begin{proof}
First of all, the locally free sheaf $L^m\otimes F^e_*L^i$ on $E$ is 
generated by its global sections if $m\ge 1$, or $m=0$ and $q<di$. 
To see this, let $P\in E$ and consider the exact sequence 
\begin{eqnarray} &
0\to L^m(-P)\otimes F^e_*L^i \to L^m\otimes F^e_*L^i 
                               \to \kappa(P)\otimes L^m\otimes F^e_*L^i \to 0.
\end{eqnarray}
Since 
$h^1(L^m(-P)\otimes F^e_*L^i)=h^1(L^{qm+i}(-qP))=h^0(L^{-qm-i}(qP))=0$
by the assumption, the induced map $H^0(E,L^m\otimes F^e_*L^i) \to 
H^0(E,\kappa(P)\otimes L^m\otimes F^e_*L^i)$ is surjective, i.e., 
$L^m\otimes F^e_*L$ is generated by its global sections at $P\in E$. 
Hence 
\begin{eqnarray*}
f^{\star}[R^{1/q}]_{i/q\!\!\!\mod\Z} 
 & = & \Im([R^{1/q}]_{i/q\!\!\!\mod\Z}\otimes_R\cO_{\widetilde X}
                                      \to F^e_*\cO_{\widetilde X})     \\
 & = & \Im\displaystyle{\left(
         \bigoplus_{m\ge 0} H^0(E,L^m\otimes F^e_*L^i)\otimes_k\cO_E 
         \stackrel{\alpha}{\lra} \bigoplus_{m\ge 0} L^m\otimes F^e_*L^i 
                                                             \right)} \\
 & = & \Im(H^0(E,F^e_*L^i)\otimes\cO_E\stackrel{\alpha_0}{\lra} F^e_*L^i)
                        \oplus \bigoplus_{m\ge 1} L^m\otimes F^e_*L^i \\
 & \subset & \bigoplus_{m\ge 0} L^m\otimes F^e_*L^i \cong \pi^*F^e_*L^i,
\end{eqnarray*}
where $\alpha_m$ ($m\ge 0$) is the graded part of the map $\alpha$ 
of degree $m$, and in particular, 
$f^{\star}[R^{1/q}]_{i/q\!\!\mod\Z} \cong \pi^*F^e_*L^i$ if $q<di$. 
Since $\pi^*F^e_*L^i$ is a locally free $\cO_{\widetilde X}$-module, 
we consider the case $q>di$. Since $\alpha_m$ is surjective for $m\ge 1$, the 
$\cO_{\widetilde X}$-module $\Coker(\alpha)=\Coker(\alpha_0)$ is regarded 
as a coherent sheaf on the exceptional curve $E\subset\widetilde{X}$ of $f$. 
Then we will see: 

\medskip\noindent\textbf{Claim.}
$\Coker(\alpha)=\Coker(\alpha_0)$ is a locally free sheaf on $E$, so 
that it has depth 1 as an $\cO_{\widetilde X}$-module at each point on 
$E\subset \widetilde{X}$. 

\medskip
To prove the claim, note that $h^0(F^e_*L^i)=h^0(L^i)=di$ by Riemann-Roch 
and that $F^e_*L^i$ is a locally free sheaf on $E$ of rank $q$, so that 
the rank of $\Coker(\alpha)=\Coker(\alpha_0)$ as an $\cO_E$-module is at 
least $q-di$. On the other hand, since $H^0(E,\cO_E(-P)\otimes F^e_*L^i) =
H^0(E,L^i(-qP))=0$ by our assumption, the cohomology long exact sequence 
of (3) for $m=0$ turns out to be 
$$0 \to H^0(E,F^e_*L^i) \to \kappa(P)\otimes F^e_*L^i 
                        \to H^1(E,\cO_E(-P)\otimes F^e_*L^i) \to 0,$$ 
from which we see that the minimal number of local generators of 
$\Coker(\alpha)$ is $\dim\Coker(\alpha_0)\otimes\kappa(P)=q-di$. 
Comparing the rank and the minimal number of local generators, we conclude 
that $\Coker(\alpha)=\Coker(\alpha_0)$ is a locally free sheaf on $E$ of 
rank $q-di$. 

Now we have an exact sequence of $\cO_{\widetilde X}$-modules 
$$
0\to f^{\star}[R^{1/q}]_{i/q\!\!\!\mod\Z} \to \pi^*F^e_*L^i 
                                          \to\Coker(\alpha)\to 0,$$
in which $\pi^*F^e_*L^i$ and $\Coker(\alpha)$ have depth 2 and 1, 
respectively. Thus the depth of $f^{\star}[R^{1/q}]_{i/q\!\!\mod\Z}$
is 2, so that it is locally free on $\widetilde X$. 
\end{proof}

\begin{rem}
In the case where $1\le i\le q-1$ and $q=di$, a similar argument as 
in the proof of Lemma \ref{lem1} shows that 
$f^{\star}[R^{1/q}]_{i/q\!\!\mod\Z}$ is \emph{not} flat at 
$P\in E\subset \widetilde{X}$ if and only if $L^i\cong \cO_E(qP)$. 
\end{rem}

\begin{crl}\label{ellip0}
If $q=p^e>1$ and $d=-E^2$ is not a power of the characteristic $p$, 
then $\widetilde X$ is the normalization of the blowup $\Bl_{N_q}(X)$ 
of $X=\Spec R$ at the $R$-module 
$N_q=\bigoplus_{i=1}^{q-1} [R^{1/q}]_{i/q\!\!\mod\Z}$. 
\end{crl}

\begin{proof}
First we will see that $N_q$ is not flat if $q=p^e>1$. For, if $N_q$ is 
flat, then the $\cO_{X,x}$-module $\cO_{X,x}^{1/q}$ has a free summand 
of rank at least $q(q-1)$. However, the rank of the free summand of 
$\cO_{X,x}^{1/q}$ is exactly equal to $1$, since $\cO_{X,x}$ is a Gorenstein 
F-pure local ring with isolated non-F-regular locus; see Aberbach--Enescu 
\cite{AE} and Sannai--Watanabe \cite[Theorem 5.1]{SW}. 

Now by Lemma \ref{lem1}, the minimal resolution 
$f \colon \widetilde X \to X$ is a flattening of $N_q$, so it factors as 
$f\colon\widetilde{X}\stackrel{g}{\to}\Bl_{N_q}(X)\stackrel{h}{\to}X$. 
Since $N_q$ is not flat and $X$ is normal, $h$ is not an isomorphism 
and has an exceptional curve, which is equal to $g(E)$. Hence $g$ is 
finite (and birational), so that $\widetilde X$ is the normalization 
of $\Bl_{N_q}(X)$. 
\end{proof}

Next we consider the structure of $f^{\star}[R^{1/q}]_{0\!\!\mod\Z}$, 
which depends on whether $R$ is F-pure or not. This is equivalent to 
saying whether the elliptic curve $E$ is ordinary or supersingular, 
since the section ring $R=R(E,L)$ is F-pure if and only if $E=\Proj R$ 
is F-split. 

\subsection{The F-pure case.}
We first consider the case where $R$ is F-pure, or equivalently, 
$E$ is an ordinary elliptic curve. In this case, given a fixed point 
$P_0 \in E$ as the identity element of the group law of $E$, there 
are exactly $q=p^e$ distinct $q$-torsion points $P_0,\dots,P_{q-1}$. 
In other words, there are exactly $q$ non-isomorphic $q$-torsion 
line bundles $L_0,\dots,L_{q-1}\in\mathrm{Pic}^{\circ}(E)$ given 
by $L_i=\cO_E(P_i-P_0)$. Then $F^e_*\cO_E$ splits into line bundles 
as 
\begin{eqnarray}
F^e_*\cO_E \cong \bigoplus_{i=0}^{q-1} L_i.
\end{eqnarray}
Indeed, since $\cO_E$ is a direct summand of $F^e_*\cO_E$ by 
F-splitting, each $L_i$ is a direct summand of 
$L_i\otimes F^e_*\cO_E \cong F^e_*F^{e*} L_i 
                      \cong F^e_*(L_i^q) \cong F^e_*\cO_E$; 
cf.\ \cite{At}. 

\begin{lmm}\label{lem2}
Let $E$ be an ordinary elliptic curve. 
\begin{enumerate}
\renewcommand{\labelenumi}{(\arabic{enumi})}
\item Suppose that $d=1$ and choose the identity element $P_0\in E$ 
so that $L\cong\cO_E(P_0)$. Then $f^{\star}[R^{1/q}]_{0\!\!\mod\Z}$ 
is not flat exactly at the $q-1$ distinct $q$-torsion points 
$P_1,\dots,P_{q-1} \in E\subset\widetilde{X}$ other than $P_0$. 
Moreover, $[R^{1/q}]_{0\!\!\mod\Z}$ is flattened by blowing up 
the points $P_1,\dots,P_{q-1}$. 
\item If $d\ge 2$, then $\widetilde X$ is a flattening of 
$[R^{1/q}]_{0\!\!\mod\Z}$. 
\end{enumerate}
\end{lmm}

\begin{proof}
Corresponding to the splitting of $F^e_*\cO_E$ as in the formula (2) 
above, the $R$-module $[R^{1/q}]_{0\!\!\mod\Z}$ has a splitting 
$[R^{1/q}]_{0\!\!\mod\Z} \cong \bigoplus_{i=0}^{q-1} N_i$ into 
$q$ non-isomorphic reflexive $R$-modules $R=J_0,J_1,\dots,J_{q-1}$ 
of rank 1, where 
$$J_i=\Gamma_*(L_i):=\bigoplus_{m\in\Z} H^0(E,L_i\otimes L^m)
                    =\bigoplus_{m\ge 0} H^0(E,L_i\otimes L^m).$$ 

In case (1) where $d=1$, it is sufficient to show the following 

\medskip\noindent\textbf{Claim.}
For $i=1,\dots,q-1$, $f^{\star}J_i$ is not flat 
exactly at the single point $P_i \in E\subset \widetilde{X}$. 
If $\sigma_i\colon \widetilde{X}_i \to \widetilde{X}$ is the blowup 
at $P_i$, then $(f\circ\sigma_i)^{\star}J_i$ is invertible. 

\medskip
To prove the claim, note that $\deg L=1$ and $\deg L_i=0$. Then the 
following holds for the linear system $|L_i\otimes L^m|$ on $E$: 
$|L_i|=\emptyset$, $|L_i\otimes L|=$ Bs$|L_i\otimes L|=\{P_i\}$ and 
$|L_i\otimes L^m|$ is base point free for $m\ge 2$. Hence, as in the 
proof of the previous lemma, 
\begin{eqnarray*}
f^{\star}J_i 
  & = & \Im(J_i\otimes_R\cO_{\widetilde X}\to F^e_*\cO_{\widetilde X}) \\
  & = & \Im\displaystyle{\left(
           \bigoplus_{m\ge 0} H^0(E,L_i\otimes L^m)\otimes_k\cO_E
           \to \bigoplus_{m\ge 0}L_i\otimes L^m\right)}              \\
  & = & L_i\otimes L(-P_i) \oplus \bigoplus_{m\ge 2} L_i\otimes L^m 
  \subset \bigoplus_{m\ge 1} L_i\otimes L^m 
  \cong \cO_{\widetilde X}(-E) \otimes \pi^*L_i, 
\end{eqnarray*}
where $L_i\otimes L(-P_i) \cong \cO_E \subset L_i\otimes L$ is the graded 
part of degree $m=1$. We therefore have the following exact sequence of 
$\cO_{\widetilde X}$-modules
$$
0\to f^{\star}J_i \to 
     \cO_{\widetilde X}(-E)\otimes\pi^*L_i \to \kappa(P_i)\to 0,$$ 
which tells us that 
$f^{\star}J_i=\I_{P_i}\cdot\cO_{\widetilde X}(-E)\otimes\pi^*L_i$,
where $\I_{P_i}$ is the ideal sheaf defining the closed point 
$P_i\in\widetilde{X}$. Now the claim follows immediately. 

(2) If $\deg L\ge 2$, then the same argument as in (1) shows that 
$f^{\star}J_i$ is isomorphic to $\cO_{\widetilde X}(-E)\otimes\pi^*L_i$, 
which is invertible. 
\end{proof}

\if
We also need the following general lemma. 

\begin{lmm}\label{trivial}
Let $Y$ be a smooth surface, $C$ a complete curve on $Y$ and let 
$\mathcal F$ be a globally generated locally free sheaf of rank 
$r$ on $Y$ such that $c_1(\mathcal{F})C=0$. Then there exists an 
open neighborhood $U \subset Y$ of $C$ such that 
$\mathcal{F}|_U \cong \cO_U^{\oplus r}$. 
\end{lmm}


\begin{proof}
Fix any point $P\in C$ and choose global sections $s_1,\dots,s_r\in
H^0(Y,\mathcal{F})$ that generate $\mathcal F$ at $P$. We consider 
the induced monomorphisms $s\colon \cO_Y^{\oplus r}\to \mathcal{F}$ 
and $\bar{s} \colon \cO_C^{\oplus r} \to \mathcal{F}\otimes\cO_C$ 
sitting in the following commutative diagram. 
$$\begin{array}{ccc}
  \cO_Y^{\oplus r} &    \stackrel{s}{\lra}    & \mathcal{F} \\
    \downarrow    &                          & \downarrow  \\
  \cO_C^{\oplus r} & \stackrel{\bar{s}}{\lra} & \mathcal{F}\otimes\cO_C
\end{array}
$$
If the map $\bar{s}$ is not surjective, then its determinant map 
$\det\bar{s} \colon \cO_C \hookrightarrow \det\mathcal{F}\otimes\cO_C$ 
is not surjective, so that $c_1(\mathcal{F})C>0$, a contradiction. 
Thus the map $\bar{s}$ is surjective, and so is the map $s$ at any 
point of $C$ by Nakayama's lemma. Hence there exists an open neighborhood 
$U$ of $C$ such that $s|_U \colon \cO_U^{\oplus r} \to \mathcal{F}|_U$ 
is an isomorphism. 
\end{proof}
\fi

We now state a structure theorem for F-blowups of F-pure 
$\widetilde{E}_8$-singularities, that is, F-pure simple elliptic 
singularities with $E^2=-1$. 

\begin{thm}\label{ellip1}
Let $(X,x)$ be an F-pure simple elliptic singularity with the elliptic 
exceptional curve $E$ on the minimal resolution $\widetilde X$ such 
that $E^2=-1$. Let $P_0,\dots,P_{q-1}\in E$ be the $q=p^e$ distinct 
$q$-torsion points on $E\subset\widetilde{X}$, where the identity 
element $P_0$ is chosen so that $\cO_{\widetilde X}(-E)\otimes\cO_E 
\cong\cO_E(P_0)$, and let $Z=\{P_1,\dots,P_{q-1}\}\subset\widetilde{X}$. 
Then for any $e\ge 1$, the normalization of the $e$-th F-blowup 
$\FB_e(X)$ coincides with the blowup $\Bl_Z(\widetilde{X})$ of 
$\widetilde X$ at the $q$-torsion points other than $P_0$. 
\par
In particular, the $e$-th F-blowup of $X$ is not dominated by the 
minimal resolution of the singularity $(X,x)$, and the monotonic 
sequence of F-blowups {\rm (}see \cite{Y3}{\rm )}
$$\cdots\to \FB_e(X) \to\cdots\to \FB_2(X) \to \FB_1(X) \to X$$ 
does not stabilize.
\end{thm}

\begin{proof}
Since $N_q=\bigoplus_{i=1}^{q-1}[R^{1/q}]_{i/q\!\!\mod\Z}$ is a direct 
summand of $R^{1/q}$ as an $R$-module, we have a morphism 
$\FB_e(X) \to \Bl_{N_q}(X)$ over $X$. If we denote the normalization 
of $\FB_e(X)$ by $\widetilde{\FB}_e(X)$, then we have a morphism 
$\varphi \colon \widetilde{\FB}_e(X) \to \widetilde{X}$ by Corollary 
\ref{ellip0}. On the other hand, since $\Bl_Z(\widetilde{X})$ is a flattening 
of $R^{1/q}$ by Lemmas \ref{lem1} and \ref{lem2}, we have a morphism 
$\Bl_Z(\widetilde{X}) \to \FB_e(X)$ over $X$, which induces  
$\psi \colon \Bl_Z(\widetilde{X}) \to \widetilde{\FB}_e(X)$. 
Thus the blowup $\pi \colon \Bl_Z(\widetilde{X}) \to \widetilde{X}$ 
at $Z \subset \widetilde{X}$ factors as
$$
\pi=\varphi\circ\psi \colon \Bl_Z(\widetilde{X}) \stackrel{\psi}{\lra}
                      \widetilde{\FB}_e(X) \stackrel{\varphi}{\lra} \widetilde{X}.
$$
Since $f^{\star}R^{1/q}$ is not flat exactly at $Z=\{P_1,\dots,P_{q-1}\}$ 
by Lemma \ref{lem2}, $\varphi$ has an exceptional curve over every 
$P_i$ and $\psi$ is finite (and birational), by the same argument as in 
the proof of Corollary \ref{ellip0}. Since $\widetilde{\FB}_e(X)$ is normal, 
$\psi$ is an isomorphism, that is, 
$\Bl_Z(\widetilde{X}) \cong \widetilde{\FB}_e(X)$ as required.
\end{proof}

The above theorem has nothing to say about the normality of the F-blowups. 
Let's take a look at Macaulay2 computation. 

\begin{exa}
\textrm{F}rom \cite[Cor.\ 4.3]{Hi}, the variety
\[
X = \Spec \FF_{2}[x,y,z]/(y^{2}+x^{3}+xyz+z^{6}) 
\]
has a simple elliptic singularity of type $\tilde E_{8}$. 
Moreover from Fedder's criterion \cite{F}, this is F-pure.
Note that since F-blowups are compatible with extensions of perfect fields \cite{Y1},
no problem occurs by that the base field is not algebraically closed.
By Macaulay2 computation, one can check the following: 
The first F-blowup  $\FB_{1}(X)$ is non-normal and its exceptional set consists of
 two projective lines $E_{1}$ and $E_{2}$ which intersects transversally at one point.
The normalization  $\widetilde{\FB}_{1}(X)$ of $\FB_{1}(X)$ is smooth. 
 The inverse image of $E_{1}$ in
$\widetilde{\FB}_{1}(X)$ is a smooth elliptic curve, 
 which agrees with Theorem \ref{ellip1}.
In particular, this experimental result shows that the normalization in the theorem is really necessary.
\end{exa}


Next we consider the case where $E^2\le-2$. 

\begin{thm}\label{ellip2}
Let $(X,x)$ be an F-pure simple elliptic singularity with the elliptic 
exceptional curve $E$ on the minimal resolution $\widetilde X$ such 
that $E^2\le-2$. Assume further that $d=-E^2$ is not a power of the 
characteristic $p$. Then $\widetilde X$ is the normalization of the 
$e$-th F-blowup $\FB_e(X)$ for all $e\ge 1$. 

Moreover, if $E^2\le-3$, then $\widetilde{X}\cong\FB_e(X)$ for all 
$e\ge1$. 
\end{thm}

\begin{proof}
Since $\widetilde X$ is a flattening of $R^{1/q}$ by Lemmas \ref{lem1} 
and \ref{lem2}, we see that $\widetilde X$ is the normalization of 
$\FB_e(X)$ as in the proof of Corollary \ref{ellip0}. 

To deduce a stronger conclusion in the special case $E^2\le-3$, 
we need the following 

\begin{lmm}[Mumford \cite{M}]
Let $V$ be a projective variety, $\F$ a coherent sheaf on $V$ and 
let $L$ be a line bundle on $V$ generated by its global sections. 
Suppose that $H^i(V,\F\otimes L^{-i})=0$ for all $i>0$. Then the 
natural map 
$$H^0(V,\F)\otimes H^0(V,L)^{\otimes n} \to H^0(V,\F\otimes L^n)$$ 
is surjective for all $n\ge1$. 
\end{lmm}

\begin{lmm}\label{proj-normal}
Let $H_1,\dots,H_n$ be line bundles on an elliptic curve $E$ 
of $\deg H_i \ge 3$ for $i=1,\dots,n$. Then the natural map 
$$H^0(E,H_1)\otimes\cdots\otimes H^0(E,H_n) 
                   \to H^0(E,H_1\otimes\cdots\otimes H_n)$$ 
is surjective.
\end{lmm}

\begin{proof}
The case $n\ge 3$ is easily reduced to the case $n=2$ by 
induction on $n$, so let $n=2$. If $\deg H_1>\deg H_2$, then 
$H^1(E,H_1\otimes H_2^{-1})=0$, so that the surjectivity of 
the map 
$H^0(E,H_1)\otimes H^0(E,H_2) \to H^0(E,H_1\otimes H_2)$
immediately follows from Mumford's lemma. Suppose that 
$\deg H_1=\deg H_2$ and let $L=H_2-P$ for any fixed point $P\in E$. 
Then $L$ is globally generated since $\deg L\ge 2$, and 
$H^1(E,H_1\otimes L^{-1})=0$ since $\deg(H_1\otimes L^{-1})=1>0$. 
Hence the map 
$H^0(E,H_1)\otimes H^0(E,L) \to H^0(E,H_1\otimes L)$ 
is surjective by Mumford's lemma. We now consider the following 
commutative diagram with exact rows: 
$$\begin{array}{cccccc}
0     \to\!\! & \!\! H^0(H_1) \otimes H^0(L)  \!\! & 
  \!\!\to\!\! & \!\! H^0(H_1)\otimes H^0(H_2) \!\! & 
  \!\!\to\!\! & \!\! H^0(H_1)\otimes H^0(H_2\otimes\kappa(P))\to 0 \\
  & \downarrow & & \downarrow & & \downarrow \phantom{\to 0} \\
0 \to\!\! & H^0(H_1\otimes L) & \to & H^0(H_1\otimes H_2) & 
  \to & H^0(H_1\otimes H_2\otimes\kappa(P)) \, ,\phantom{\to 0}
\end{array}
$$
where we have just verified the surjectivity of the vertical map 
on the left, and the vanishing of the right upper corner comes 
from $H^1(E,L)=0$. So, to prove the required surjectivity of the 
vertical map in the middle, it suffices to show that the vertical 
map on the right is surjective, by the five-lemma. This map is 
factorized as 
$$H^0(H_1)\otimes H^0(H_2\otimes\kappa(P)) \stackrel{\alpha}{\to}
  H^0(H_1\otimes\kappa(P))\otimes H^0(H_2\otimes\kappa(P)) 
  \stackrel{\beta}{\to} H^0(H_1\otimes H_2\otimes\kappa(P)).$$ 
Here $\alpha$ is surjective because of the vanishing 
$H^1(E,H_1(-P))=0$, and $\beta$ is identified with the 
multiplication map $k^{\otimes 2} \stackrel{\sim}{\lra} k$, which 
is clearly surjective. Thus $\beta\circ\alpha$ is surjective, and 
the lemma is proved. 
\end{proof}

We continue the proof of Theorem \ref{ellip2} in the case $E^2\le-3$. 
Consider the decomposition (2) of $F^e_*\cO_E$ into $q=p^e$-torsion 
line bundles $\cO_E=L_0,L_1,\dots,L_{q-1}$ on $E$. We fix any $i$ 
with $0<i\le q-1$ and let $I\subset R$ be an ideal isomorphic to 
the reflexive $R$-module 
$$J_i=\Gamma_*(L_i)=\bigoplus_{n\ge 1} H^0(E,L_i\otimes L^n)t^n$$ 
of rank 1, which is a non-trivial $R$-summand of $R^{1/q}$. Then 
the minimal resolution $f\colon\widetilde{X}\to X=\Spec R$ is 
factorized as 
$$f\colon \widetilde{X} \to \FB_e(X) \to \Bl_I(X) \to X,$$ 
where the blowup $\Bl_I(X)=\Proj R[It]$ of $X$ with respect to 
the ideal $I$ has an exceptional curve that is the image of 
$E \subset \widetilde{X}$, since $I\cong J_i$ is not a flat 
$R$-module. It follows that $\widetilde X$ is the normalization 
of $\Bl_I(X)$. So, to prove the theorem, it is sufficient to 
show that the Rees algebra $R[It]$ is normal. 

To prove the normality of $R[It]=\bigoplus_{m\ge 0}I^mt^m$, note 
that its normalization is 
$$\widetilde{R[It]} = \bigoplus_{m\ge 0} \overline{I^m}t^m,$$ 
where $\overline{I^m} \subseteq R$ is the integral closure of 
the ideal $I^m$; see \cite{L}. Note also that 
$$I\cO_{\widetilde X} \cong f^{\star}J_i 
                    \cong \bigoplus_{n\ge 1} (L_i\otimes L^n)t^n
                    \cong \cO_{\widetilde X}(-E)\otimes\pi^*L_i 
$$
is an invertible sheaf on $\widetilde X$ by Lemma \ref{lem2}, 
so that 
$$\overline{I^m} 
  \cong H^0(\widetilde{X},\cO_{\widetilde X}(-mE)\otimes\pi^*L_i^m)
  \cong \bigoplus_{n\ge m} H^0(E,L_i^m\otimes L^n)t^n
$$
for all $m\ge 1$. Now, since $\deg L\ge 3$, we can apply Lemma 
\ref{proj-normal} to $H_1=\cdots=H_m:=L_i\otimes L$ and 
$H_{m+1}=\cdots=H_n:=L$ to obtain the surjectivity of the map 
$$H^0(E,L_i\otimes L)^{\otimes m}\otimes H^0(E,L)^{\otimes n-m} 
                                  \to H^0(E,L_i^m\otimes L^n)$$ 
for all $n\ge m\ge 1$. This implies that the multiplication map 
$\overline{I}^{\otimes m} \to \overline{I^m}$ is surjective in 
all degree $n$. Since $I=\overline{I}$ is integrally closed, we 
conclude that $I^m = \overline{I^m}$, from which the normality 
of the Rees algebra $R[It]$ follows. 
\end{proof}

\begin{exa}
Let 
\[
X = \Spec \FF_{2}[x,y,z]/(y^2+xyz+x^3z+xz^3).
\]
Again from \cite[Cor.\ 4.3]{Hi} and Fedder's criterion, 
$X$ has an F-pure simple elliptic singularity of type $\tilde E_{7}$ at the origin.
The exceptional set of $\FB_{1}(X)$ consists of three projective lines.
It shows that  
it is necessary  to suppose in Theorem \ref{ellip2}
 that $d=-E^{2}$ is not a power of $p$. 
 The normalization of $\FB_{1}(X)$ is smooth.
\end{exa}

\begin{exa}
The variety
\[
X= \Spec \FF_{2}[x,y,z]/(y^2z+xyz+x^3+z^3)
\]
 has an F-pure simple elliptic singularity of type $\tilde E_{6}$.
By Macaulay2 computations, we can see that $\FB_{1}(X)$ is smooth and 
the exceptional set is a smooth elliptic curve, as expected from Theorem \ref{ellip2}.
\end{exa}

\subsection{The non-F-pure case.}
Now we consider the structure of $f^{\star}[R^{1/q}]_{0\!\!\mod\Z}$ 
assuming that $R$ is not F-pure, or equivalently, $E$ is a supersingular 
elliptic curve. In this case $E$ has no non-trivial $q$-torsion point 
under the group law. Then, contrary to the F-pure case, $F^e_*\cO_E$ 
turns out to be indecomposable as we will see below. 

For any elliptic curve $E$ and an integer $r>0$, there exists an 
indecomposable vector bundle $\F_r$ on $E$ of rank $r$ and degree 
zero with $h^0(\F_r)=1$, determined inductively by $\F_1=\cO_E$ and 
the unique non-trivial extension 
\begin{eqnarray} & 
0 \to \F_{r-1} \to \F_r \to \cO_E \to 0. 
\end{eqnarray}
Note that $\F_r$ is self-dual and (3) is the dual sequence of 
that in \cite[Theorem 5]{At}. 

\begin{lmm}[cf.\ Atiyah \cite{At}, Tango \cite{T}]\label{ss-indec}
If $E$ is a supersingular elliptic curve, then $F^e_*\cO_E\cong\F_q$ 
for all $q=p^e$. 
\end{lmm}

\begin{proof}
Let $F^e_*\cO_E=\E_1\oplus\cdots\oplus\E_n$ be the decomposition 
of $F^e_*\cO_E$ into indecomposable bundles $\E_i$ of rank $r_i$ 
and degree $d_i$. Then $d_1+\cdots+d_n=\chi(F^e_*\cO_E)=0$ by 
Riemann--Roch. Pick a non-trivial line bundle $L$ of degree zero. 
Then 
$\sum_{i=1}^nh^0(\E_i\otimes L)=h^0(L\otimes F^e_*\cO_E)=h^0(L^q)=0$, 
since there is no non-trivial $q$-torsion line bundle on a 
supersingular elliptic curve. Hence $d_i=\deg(\E_i\otimes L)\le 0$ 
for all $i=1,\dots,n$. Thus the indecomposable summands $\E_i$ of 
$F^e_*\cO_E$ have degree $d_i=0$, and exactly one of them, say $\E_1$, 
 has a non-zero global section since $h^0(F^e_*\cO_E)=1$. Then by 
Atiyah \cite[Theorem 5]{At}, we have $\E_1 \cong \F_{r_1}$ and 
$\E_i \cong \F_{r_i} \otimes L_i$ for $i=2,\dots,n$, where 
$L_2,\dots,L_n$ are non-trivial line bundles of degree zero. 
Suppose that $n\ge 2$. Then $L_2^{-1}\otimes F^e_*\cO_E$ has a 
non-zero global section since its direct summand $\F_{r_2}$ does. 
On the other hand, however, 
$H^0(E,L_2^{-1}\otimes F^e_*\cO_E)=H^0(E,L_2^{-q})=0$ 
since $L_2$ is not a $q$-torsion line bundle by our assumption. 
We thus conclude that $n=1$, i.e., $F^e_*\cO_E \cong \F_q$. 
\end{proof}

Now for each $r$, we consider the graded $R$-module 
$$M_r = \bigoplus_{n\ge 0} H^0(E,\F_r\otimes L^n)t^n$$ 
and regard its torsion-free pullback $\widetilde{M_r}=f^{\star}M_r$ 
to the minimal resolution $\widetilde X$ of $X=\Spec R$ as a subsheaf 
of 
$$\M_r=\bigoplus_{n\ge 0} (\F_r\otimes L^n)t^n.$$ 

To obtain an information on the flattening of $R^{1/q}$, we consider 
the torsion-free pullback $f^{\star}M_r$ of $M_r$ to the minimal 
resolution, because $[R^{1/q}]_{0\!\!\mod\Z} \cong M_q$ by Lemma 
\ref{ss-indec}. 

\medskip\noindent
4.2.1.~\textit{Non-F-pure $\widetilde{E}_8$-singularities}. 
We first consider the case of  $\widetilde{E}_8$-singularities, that is, 
the case $\deg L=-E^2=1$. In this case, $L\cong\cO_E(P_0)$ for a point 
$P_0\in E$. 

We fix any point $P\in E$ and let $V \subset E$ be a sufficiently small 
open neighborhood $V$ of $P$ on which $L$ and $\F_r$ trivialize. 
We choose a local basis $e_1,\dots,e_r$ of $\F_r$ on $V$ inductively 
as follows. For $r=1$, let $e_1$ be a (local) basis of $\F_1=\cO_E$ 
corresponding to its global section $1\in H^0(E,\cO_E)$. For $r\ge 2$, 
we think of $\F_{r-1}$ as a subbundle of $\F_r$ via the exact sequence 
(3), and extend the local basis $e_1,\dots,e_{r-1}$ of $\F_{r-1}$ on $V$ 
to a local basis $e_1,\dots,e_r$ of $\F_r$.  

Let $U=\pi^{-1}V \subset \widetilde{X}$. Then, with the local 
trivialization $L|_V\cong\cO_V$ and 
$\F_r|_V\cong\bigoplus_{i=1}^r\cO_Ve_i \cong \cO_V^{\oplus r}$ 
as above, we have 
$$
\M_r|_U \cong \bigoplus_{i=1}^r\cO_Ue_i \cong \cO_U^{\oplus r},
$$
where $\cO_U = \bigoplus_{n\ge 0}(L|_V)^nt^n 
            \cong \bigoplus_{n\ge 0}\cO_Vt^n = \cO_V[t]$.
Note that the fiber coordinate $t$ and a regular parameter $u$ at 
$P \in E$ form a system of coordinates of $U$. With this notation 
we shall express generators of the $\cO_U$-module $\widetilde{M}_r|_U 
\subseteq \M_r|_U$, which come from homogeneous elements of the 
graded $R$-module $M_r$. 

First note that the degree zero piece $[M_r]_0=H^0(E,\F_r)=H^0(E,\F_1)$ 
of $M_r$ is a one-dimensional $k$-vector space, so that its contribution 
to the generation of $\widetilde{M}_r|_U$ is just $e_1$. It is also easy 
to see that the graded parts of $\widetilde{M}_r|_U$ and $\mathcal{M}_r|_U$ 
coincide in degree $\ge 2$ and are generated by $t^2e_1,\dots,t^2e_r$, 
since $\F_r\otimes L^n$ is generated by global sections for $n\ge 2$. 
It remains to consider the contribution of the degree one piece  
$[M_r]_1=H^0(E,\F_r\otimes L)t$ to the generation of $\widetilde{M}_r|_U$. 
To this end, note that we have an exact sequence 
$$
0 \to H^0(E,\F_i\otimes L) \to H^0(E,\F_{i+1}\otimes L) \to H^0(E,L) \to 0 
$$
for $1\le i\le r-1$, via which we regard $H^0(E,\F_i\otimes L)$ as 
a subspace of $H^0(E,\F_r\otimes L)$. Then, since $h^0(\F_i\otimes L)=i$ 
by Riemann-Roch, we  can choose a basis $s_1,\dots,s_r$ of 
$H^0(E,\F_r\otimes L)$ so that 
$s_1,\dots,s_i$ form a basis of $H^0(E,\F_i\otimes L)$ for $1\le i\le r$. 
It also follows from exact sequence $(3)\otimes L$ that the global 
sections $s_1,\dots,s_i$ generate $\F_i\otimes L$ on $E\setminus\{P_0\}$, 
so that they give a basis of $\F_i\otimes L\otimes K$ as a vector space 
over the function field $K$ of $E$. On the other hand, $e_1,\dots,e_i$ can 
also be viewed as a basis of $\F_i\otimes L\otimes K \cong K^{\oplus i}$ 
under the local trivialization 
$\F_i\otimes L|_V \cong \bigoplus_{j=1}^i\cO_Ve_i \cong \cO_V^{\oplus i}$ 
induced from $\F_i|_V\cong\cO_V^{\oplus i}$ and $L|_V\cong\cO_V$. We will 
compare the basis consisting of $s_i\otimes 1$ and the standard basis 
$e_1,\dots,e_r$ of $\F_r\otimes L\otimes K \cong K^{\oplus r}$ using 
the following commutative diagram with exact rows: 
$$\begin{array}{ccccccc}
0 \to & \!\! H^0(\F_{i-1}\otimes L)\otimes\cO_V \!\! &
  \!\! \to \!\! & \!\! H^0(\F_i\otimes L)\otimes\cO_V \!\! & 
  \!\! \to \!\! & \!\! H^0(L)\otimes\cO_V \!\! & \to 0 \\
        & \downarrow & & \downarrow & & \downarrow & \\
0 \to & \F_{i-1}\otimes L|_V & 
   \to & \F_i\otimes L|_V & \to & L|_V & \to 0 \\
        & \downarrow_{{}^{\cong}} & & \downarrow_{{}^{\cong}} & 
        & \downarrow_{{}^{\cong}} & \\
0 \to & \cO_V^{\oplus i-1} & \to & \cO_V^{\oplus i} & \to & \cO_V & \to 0
\end{array}
$$

Suppose now that $P=P_0$. Since Bs$|L|=\{P_0\}$, we may choose a 
regular parameter $u$ at $P_0\in E$ so that $s_1\otimes 1=u$. 
It then follows from the above diagram that  
$$s_i\otimes 1=ue_i+\sum_{j=1}^{i-1}a_{i,j}e_j,$$
where $a_{ij}$'s are local regular functions on $V$. 
We claim that we can replace $s_1,\dots,s_r$ so that they satisfy the 
condition: 
\begin{eqnarray}
u|a_{i,j} \text{ for } 1\le j\le i-2 \text{ but } 
a_{i,i-1} \text{ is not divisible by } u. & 
\end{eqnarray}

To prove the claim, there is nothing to do with $i=1$. So let $i=2$ 
and suppose $u|a_{2,1}$. We consider a $k$-linear map 
$H^0(E,L) \to H^0(E,\F_2\otimes L)$ given by $s_1\mapsto s_2$, 
which gives rise to a $K$-linear map
 $K \cong L\otimes K \to \F_2\otimes L\otimes K \cong K^2$ 
sending 
$1=u^{-1}(s_1\otimes 1) \mapsto u^{-1}(s_2\otimes 1)=e_2+(a_{21}/u)e_1$. 
Since $a_{21}/u\in\cO_V$, this gives a splitting of the surjective map 
$\cO_V^{\oplus 2} \cong \F_2\otimes L|_V \to L|_V \cong \cO_V$ at $P_0\in V$, 
as well as at any other point. Then we have a global splitting of the 
surjective map $\F_2\otimes L\to L$, contradicting to the non-triviality 
of the extension (3). Thus $a_{2,1}(P_0) \ne 0$. 
Next let $i\ge 3$. Then by induction, we may replace $s_i$ by 
$s_i-\sum_{j=1}^{i-2}(a_{i,j}(P_0)/a_{j+1,j}(P_0))s_{j+1}$ to assume that 
$u|a_{i,j}$ for $1\le j\le i-2$. It then follows that $a_{i,i-1}$ is not 
divisible by $u$ because otherwise, $s_1 \mapsto s_i$ would give a global 
splitting of $\F_i\otimes L \to L$ as above. 

Consequently, local generators of $\widetilde{M}_r$ on a neighborhood 
$U_0$ of $P_0$ are described as  
\begin{eqnarray*}
\widetilde{M}_r|_{U_0}
 & = & \cO_{U_0}\langle e_1,tue_i+a_{i,i-1}te_{i-1},t^2e_i\,|\,2\le i\le r\rangle \\
 & = & \cO_{U_0}\langle e_1,tue_i+a_{i,i-1}te_{i-1},t^2e_r\,|\,2\le i\le r\rangle,
\end{eqnarray*}
where $a_{i,i-1}(P_0) \ne 0$. Accordingly the ideal 
$\I_{\widetilde{M}_r} \subset \cO_{\widetilde X}$ defined in Section 2 has 
the following local expression:
$$\I_{\widetilde{M}_r}|_{U_0} \cong (t^r,t^{r-1}u^{r-1}) \cong (t,u^{r-1}).$$

If $P_0 \ne P \in U$ then 
$\widetilde{M_r}|_U = \cO_U\langle e_1,te_i\,|\,2\le i\le r\rangle 
                \cong \cO_U^{\oplus r}$ 
by a similar argument. 

Summarizing the argument so far, we have 

\begin{thm}\label{ellip3}
Let $(X,x)$ be a non-F-pure simple elliptic singularity with the 
elliptic exceptional curve $E$ on the minimal resolution $\widetilde X$ 
such that $E^2=-1$. Let $P_0$ be the point on $E\subset\widetilde{X}$ 
such that $\cO_{\widetilde X}(-E)\otimes\cO_E\cong\cO_E(P_0)$ and let 
$\I_e \subset \cO_{\widetilde X}$ 
be the ideal sheaf defining a fat point supported at $P_0\in\widetilde X$ 
whose local expression at $P_0$ is 
$$(\I_e)_{P_0}=(t,u^{p^e-1})$$
as above. Then for any $e\ge 1$, the blowup $\Bl_{\I_e}(\widetilde{X})$ 
of $\widetilde X$ at $\I_e$ coincides with the normalization of the 
$e$-th F-blowup $\FB_e(X)$. 
\end{thm}

\begin{proof}
We know that $Y=\Bl_{\I_e}(\widetilde{X})$ is a flattening of $R^{1/p^e}$ 
from the above argument and Corollary \ref{ellip0}. It is also easy to see 
that the exceptional curve of the blowup $\pi \colon Y \to \widetilde{X}$ 
is a single $\PP^1$. 
Then the same argument as in the proof of Theorem \ref{ellip1} shows that 
$\pi$ factors through the normalized F-blowup $\widetilde{\FB}_e(X)$ as  
$$
\pi=\varphi\circ\psi \colon Y=\Bl_{\I_e}(\widetilde{X}) \stackrel{\psi}{\lra}
                     \widetilde{\FB}_e(X) \stackrel{\varphi}{\lra} \widetilde{X}
$$
and that $\psi$ gives an isomorphism $Y \cong \widetilde{\FB}_e(X)$. 
\end{proof}

%

\begin{rem}
Theorem \ref{ellip3} says that the $e$-th normalized F-blowup 
$\widetilde{FB}_e(X)$ has the exceptional set consisting of an 
elliptic curve $E_1 \cong E$ and a smooth rational curve $E_2\cong\PP^1$, 
and has an $A_{p^e-2}$-singularity on $E_2\setminus E_1$. The theorem 
also says that $\FB_e(X)$ does not dominate $\FB_{e'}(X)$ whenever 
$e$ and $e'$ are distinct positive integers. In other words, the 
monotonicity of F-blowup sequences breaks down for non-F-pure 
$\widetilde{E}_8$-singularities; compare to the F-pure case \cite{Y3}. 
On the other hand, it again has nothing to say about the normality of 
$\FB_e(X)$. 

Let's examine our observation with Macaulay2 computation. 
\end{rem}

\begin{exa}\label{exa-1}
The variety
\[
X = \Spec \FF_{3}[x,y,z]/(x(x-z^{2})(x-2z^{2}) -y^{2}),
\]
 has a non-F-pure simple elliptic singularity of type $\tilde E_{8}$.
The exceptional set of $\FB_{1}(X)$ is the union of a smooth elliptic curve  $E_{1}$ and a projective line $E_{2}$.
By Macaulay2 computation, we could  check not that $\FB_{1}(X)$ is normal, but that $\FB_1(X)$ 
is normal at the generic points of $E_1$ and $E_2$,
and  there is a point  on $E_{2}\setminus E_{1}$  where $\FB_{1}(X)$ is normal but singular. 
The blowup of $\FB_{1}(X)$ at this point has the projective line as its exceptional locus.
It agrees with the fact that $\FB_{1}(X)$ has an $A_{1}$-singularity on $E_{2} \setminus E_{1}$ as stated in the above remark.
\end{exa}

\begin{lmm}
For $X$ as in Example \ref{exa-1}, if the Macaulay2 computation explained above
is correct, then $\FB_1(X)$ is normal.
\end{lmm}

\begin{proof}
We may replace the base field $\FF_3$ with an algebraically closed filed $k$.
Being quasi-homogeneous, $X$ has a $k^*$-action. 
From the construction or the universality, the action lifts to F-blowups of $X$.
Every point of the divisor $E_1 \subset \FB_1(X)$, which is 
a smooth elliptic curve, is  fixed by the $k^*$-action.
On the other hand, the divisor $E_2 \cong \PP^1$  has exactly two fixed points.
One is the singular but normal point mentioned above and the other is the intersection 
$E_1 \cap E_2$. Since the normal locus is open and there is the $k^*$-action, 
$\FB_1(X)$ is normal along $E_2$ possibly except at $E_1 \cap E_2$.
Therefore it is now enough to show that $\FB_1(X)$ is normal along $E_1$.
Let $\tilde E_1$ and $\tilde E_2$ be the preimages of $E_1$ and $E_2$
on the normalization $\widetilde{\FB}_1(X)$ of $\FB_1(X)$.
Then for each $i=1,2$, since $E_i$ is normal and $\FB_1(X)$ 
is normal at the generic point of $E_i$, the map $\tilde E_i \to E_i$ is an isomorphism.

Let $A$ be the complete local ring of $\FB_1(X)$ at a point $z$ on $E_1$.
Its normalization is $k[[s,t]]$. We choose local coordinates $s,t$ so that
the $k^*$-action on $k[[s,t]]$ is linear and 
locally $s=0$ defines $\tilde E_1$ and 
$t=0$ defines the only one-dimensional orbit closure passing through the point over $z$.
Then the $k^*$-action on $t$ is trivial and the one on $s$ is non-trivial.
Since $\tilde E_i \to E_i$, $i=1,2$, are isomorphisms, the composite maps
$A \hookrightarrow k[[s,t]] \to k[[s]]$ and $A \hookrightarrow k[[s,t]] \to k[[t]]$
are surjective. Therefore $A$ contains  formal power series of the forms 
\[
f= f_1s+f_2t + \text{(higher terms)}, \, (f_i \in k, \, f_1 \ne 0)
\]
and 
\[
g= g_1s+g_2t + \text{(higher terms)},\, (g_i \in k, \, g_2 \ne 0).
\]
Then by a suitable linear combination of them, we obtain a formal power series
\[
h= h_1s+h_2t + \text{(higher terms)},\, (h_i \in k, \, h_1\ne 0, \, h_{2} \ne 0)
\]
contained in $A$. Then for  $1 \ne \lambda \in k^*$, $\lambda h \in A$ has a linear part linearly independent with that of
 $h$. It follows that $A=k[[s,t]]$ and hence $\FB_{1}(X)$ is normal.
\end{proof}

\begin{exa}
The variety 
\[
X = \Spec \FF_{2}[x,y,z]/(y^2+yz^3+x^3)
\]
has a non-F-pure simple elliptic singularity of type $\tilde E_{8}$.
The Frobenius pushforward $F_{*}\cO_{X}$ of the coordinate ring decomposes into the direct sum of two modules, say $N_{1}$ and $N_{2}$.
Then $F_{*}N_{i}$ $(i=1,2)$ further decomposes as $F_{*}N_{i} = N_{i1} \oplus N_{i2}$. 
By Macaulay2 computation, we saw that the torsion-free pullbacks $\widetilde{N_{1}}$ and $\widetilde{N_{11}}$ of $N_{1}$ and $N_{11}$
are non-flat at a point and those of the others are flat. Moreover the ideals 
associated to $\widetilde{N_{1}}$ and $\widetilde{N_{11}}$ as in Proposition \ref{prp-Vi} are respectively
 of the forms $(u,v)$ and $(u,v^{3})$ around the point with respect to some local coordinates $u,v$.
The last result coincides with Theorem \ref{ellip3}.
\end{exa}


\medskip\noindent
4.2.2.~\textit{
Non-F-pure simple elliptic singularities with $E^2\le -2$}. 
In this case, we have $\deg L=-E^2 \ge 2$. Then the argument 
in (4.2.1) shows that $\widetilde{M_r}$ is flat. Thus we have 

\begin{prp}
Let $(X,x)$ be a non-F-pure simple elliptic singularity with elliptic 
exceptional curve $E$ on the minimal resolution $\widetilde X$. 
Suppose $E^2\le-2$ and $d=-E^2$ is not a power of the characteristic 
$p$. Then $\widetilde X$ is the normalization of the $e$-th F-blowup 
$\FB_e(X)$ for all $e\ge 1$. 
\end{prp}

\begin{proof}
Since $\widetilde X$ is a flattening of $R^{1/q}=M_q\oplus N_q$ by 
Lemmas \ref{lem1} and \ref{ss-indec} and (4.2.1), the proof goes 
similarly as that for Theorem \ref{ellip2}. Note that $\cO_{X,x}^{1/q}$ 
has no free summand in this case, since $\cO_{X,x}$ is not F-pure. 
\end{proof}

\noindent{\bf Question.} 
If $E^2\le-3$, then does $\widetilde{X} \cong \FB_e(X)$ hold for 
all $e\ge 1$? More generally, if $(X,x)$ is a cone singularity over 
a smooth projective curve and if $-E^2$ is sufficiently large, is 
$\FB_e(X)$ the minimal resolution? 

\begin{exa}
The variety
\[
X = \Spec \FF_{2}[x,y,z]/(y^2z+yz^2+x^3),
\]
has a non-F-pure simple elliptic singularity of type $\tilde E_{6}$.
We could check that  $\FB_{1}(X)$ is the minimal resolution.
\end{exa}

\end{document}